\pdfoutput=1
\documentclass[a4paper,10pt]{amsart}
\usepackage{
 amsmath,
 amssymb,
 mathrsfs,
 mathtools,
 fullpage,
 enumerate,
 thmtools,
 stmaryrd,
 color,
 tikz-cd
}

\usepackage[l2tabu, orthodox]{nag}
\usepackage[pagebackref]{hyperref}
\usepackage[capitalize]{cleveref}

\theoremstyle{plain}
 \newtheorem{theorem}{Theorem}[subsection]
 \newtheorem{lemma}[theorem]{Lemma}
 \newtheorem{proposition}[theorem]{Proposition}

 \newtheorem{definition}[theorem]{Definition}
 \newtheorem{notation}[theorem]{Notation}

\declaretheorem[name=Proposition,sibling=theorem,qed={\qedsymbol}]{propqed}

\theoremstyle{remark}
 \declaretheorem[name=Remark,sibling=theorem,qed={\lower-0.3ex\hbox{$\diamond$}}]{remark}
 \declaretheorem[name=Note,sibling=theorem,qed={\lower-0.3ex\hbox{$\diamond$}}]{note}

\newcommand{\tZ}{\scalebox{1.15}{$\mathtt{Z}$}}
\newcommand{\tU}{\scalebox{1.15}{$\mathtt{U}$}}
\newcommand{\tI}{\scalebox{1.15}{$\mathtt{I}$}}
\newcommand{\tM}{\scalebox{1.15}{$\mathtt{M}$}}
\newcommand{\tP}{\scalebox{1.15}{$\mathtt{P}$}}
\newcommand{\tFL}{\scalebox{1.15}{$\mathtt{FL}$}}

\newcommand{\bp}[1]{\cite[#1]{boxerpilloni20}}
\newcommand{\ds}{\diamondsuit}

\newenvironment{smatrix}{\left( \begin{smallmatrix} } {\end{smallmatrix} \right) }
\newcommand{\stbt}[4]{\begin{smatrix}#1 & #2 \\ #3 & #4\end{smatrix}}

\newcommand{\upi}{\underline{\pi}}
\newcommand{\usi}{\underline{\sigma}}
\newcommand{\uet}{\underline{\eta}}

\newcommand{\Dcris}{\mathbf{D}_{\mathrm{cris}}}

\DeclareMathOperator{\GSp}{GSp}
\DeclareMathOperator{\Kl}{Kl}
\DeclareMathOperator{\Iw}{Iw}
\DeclareMathOperator{\fs}{fs}

\DeclareMathOperator{\Hom}{Hom}
\DeclareMathOperator{\Spec}{Spec}
\DeclareMathOperator{\Fil}{Fil}
\DeclareMathOperator{\GL}{GL}
\DeclareMathOperator{\Gr}{Gr}
\DeclareMathOperator{\pr}{pr}

\DeclareMathOperator{\Gal}{Gal}
\DeclareMathOperator{\Ind}{Ind}\DeclareMathOperator{\ES}{ES}
\DeclareMathOperator{\cusp}{cusp}
\DeclareMathOperator{\Sieg}{Si}
\DeclareMathOperator{\FL}{FL}
\DeclareMathOperator{\an}{an}
\DeclareMathOperator{\HT}{HT}

\DeclareMathOperator{\Spa}{Spa}
\DeclareMathOperator{\diag}{diag}
\DeclareMathOperator{\cores}{cores}
\DeclareMathOperator{\sss}{sss}
\DeclareMathOperator{\myss}{ss}\renewcommand{\ss}{\myss}
\DeclareMathOperator{\Tor}{Tor}

\newcommand{\tor}{\mathrm{tor}}
\newcommand{\sph}{\mathrm{sph}}
\newcommand{\new}{\mathrm{new}}
\newcommand{\res}{\mathrm{res}}

\newcommand{\cB}{\mathcal{B}}
\newcommand{\cD}{\mathcal{D}}
\newcommand{\cE}{\mathcal{E}}
\newcommand{\cF}{\mathcal{F}}
\newcommand{\cG}{\mathcal{G}}
\newcommand{\cH}{\mathcal{H}}
\newcommand{\cI}{\mathcal{I}}
\newcommand{\cL}{\mathcal{L}}
\newcommand{\cM}{\mathcal{M}}
\newcommand{\cN}{\mathcal{N}}
\newcommand{\cO}{\mathcal{O}}
\newcommand{\cP}{\mathcal{P}}
\newcommand{\cQ}{\mathcal{Q}}
\newcommand{\cS}{\mathcal{S}}
\newcommand{\cT}{\mathcal{T}}
\newcommand{\cU}{\mathcal{U}}
\newcommand{\cV}{\mathcal{V}}
\newcommand{\cW}{\mathcal{W}}
\newcommand{\cX}{\mathcal{X}}
\newcommand{\cZ}{\mathcal{Z}}

\newcommand{\fa}{\mathfrak{a}}
\newcommand{\fb}{\mathfrak{b}}
\newcommand{\fS}{\mathfrak{S}}

\newcommand{\br}{\mathbf{r}}
\newcommand{\bt}{\mathbf{t}}

\newcommand{\CC}{\mathbf{C}}
\newcommand{\QQ}{\mathbf{Q}}
\newcommand{\RR}{\mathbf{R}}
\newcommand{\ZZ}{\mathbf{Z}}
\renewcommand{\AA}{\mathbf{A}}
\newcommand{\Fp}{\mathbf{F}_p}

\newcommand{\Qp}{\QQ_p}
\newcommand{\Zp}{\ZZ_p}
\newcommand{\Af}{\AA_{\mathrm{f}}}
\newcommand{\pif}{\pi_{\mathrm{f}}}

\newcommand{\dR}{\mathrm{dR}}

\newcommand{\id}{\mathrm{id}}
\renewcommand{\sp}{\operatorname{sp}}

\newcommand{\into}{\hookrightarrow}
\newcommand{\onto}{\twoheadrightarrow}

\newcommand{\htimes}{\mathop{\hat\otimes}}

\renewcommand{\le}{\leqslant}

\renewcommand{\ge}{\geqslant}
\renewcommand{\geq}{\geqslant}

\numberwithin{equation}{section}

\author{David Loeffler}
\address{David Loeffler, Mathematics Institute\\
 University of Warwick\\
 Coventry CV4 7AL, UK.}
\email{d.a.loeffler@warwick.ac.uk}
\urladdr{\href{http://orcid.org/0000-0001-9069-1877}{0000-0001-9069-1877}}

\author{Sarah Livia Zerbes}
\address{Sarah Livia Zerbes, Department of Mathematics\\
 University College London\\
London WC1E 6BT, UK.}
\email{s.zerbes@ucl.ac.uk}
\urladdr{\href{http://orcid.org/0000-0001-8650-9622}{0000-0001-8650-9622}}

\thanks{Supported by the following grants: EPSRC Standard Grant EP/S020977/1 and ERC Consolidator Grant \#101001051 ``ShimBSD'' (Loeffler).}

\title{On the Bloch--Kato conjecture for $\GSp(4)\times\GL(2)$}
\date{\today}

\begin{document}

 \begin{abstract}
  We prove the Bloch--Kato conjecture for certain critical values of degree 8 $L$-functions associated to cusp forms on $\GSp_4 \times \GL_2$. We also construct a $p$-adic Eichler--Shimura isomorphism in Hida families for $\GSp_4$, relating $H^2$ of automorphic vector bundles with $\Dcris$ of a subquotient of \'etale cohomology.
 \end{abstract}

 \maketitle

 \setcounter{tocdepth}{1}
 \tableofcontents


\section{Introduction}

 The Bloch--Kato conjecture, which relates the dimension of the Selmer group of a $p$-adic geometric Galois representation to the order of vanishing of its $L$-function, is one of the most important open problems in number theory. In a recent paper \cite{LZ20}, we proved this conjecture for the 4-dimensional Galois representations arising from automorphic representations of $\GSp_4$, under various technical hypotheses, using the ``method of Euler systems''; this relied crucially on the construction of an Euler system for $\GSp_4$ in our earlier work \cite{LSZ17} with Skinner, and the construction of a $p$-adic spin $L$-function for $\GSp_4$ in the paper \cite{LPSZ1} with Pilloni and Skinner.

 In this paper, we prove new cases of the Bloch--Kato conjecture, for the 8-dimensional tensor product Galois representation $V(\pi) \otimes V(\sigma)$ associated to a $\GSp_4$ automorphic representation $\pi$ and a $\GL_2$ automorphic representation $\sigma$. The principal ingredients are the the Euler system for these Galois representations constructed in \cite{HJS20}, and the formula proved in \cite{LZ20b-regulator} relating these Euler system classes to periods of $p$-adic modular forms for $\GSp_4 \times \GL_2$ (obtained by integrating a class in $H^2$ arising from $\pi$, restricted to $\GL_2 \times_{\GL_1} \GL_2 \subset \GSp_4$, against the product of a cusp form in $\sigma$ and a non-classical $p$-adic Eisenstein series). The main result of this paper, \cref{thm:BKconj}, proves the Bloch--Kato conjecture for a certain twist of $V(\pi)^* \otimes V(\sigma)^*$, corresponding to a critical value of the $L$-function $L(\pi \times \sigma, s)$. 

 The main new technical input needed in order to prove this theorem is to interpolate the $p$-adic automorphic periods arising from \cite{LZ20b-regulator} in $p$-adic families, with both of the weights $(r_1, r_2)$ of the $\GSp_4$ automorphic representation allowed to vary. This is not accessible by the methods of our earlier work \cite{LPSZ1}, since the version of higher Hida theory used in that paper (based on the earlier work \cite{pilloni20}) is only applicable to 1-parameter families in which $r_1$ varies for a fixed $r_2$. A similar issue arises in our earlier work \cite{LZ20}, but in that setting, we were able to bypass the problem by applying the functorial lift from $\GSp_4$ to $\GL_4$, and applying the results of \cite{DJR18, barreradimitrovwilliams} on $p$-adic $L$-functions for $\GL_{2n}$. However, this does not work for $\GSp_4 \times \GL_2$, since there appears to be no known construction of $p$-adic $L$-functions for $\GL_4 \times \GL_2$.

 We therefore develop a direct approach to interpolating these $p$-adic periods in 2-parameter families for $\GSp_4$, with both $r_1$ and $r_2$ varying, using the new ``higher Coleman theory'' introduced in \cite{boxerpilloni20}. Our main result in this direction is \cref{thm:bspairing}, whose proof occupies the majority of the present paper. This result shows that there is a well-defined pairing between the higher Coleman theory spaces for $\GSp_4$ and spaces of overconvergent modular forms for $\GL_2 \times \GL_2$; and in the final sections of the paper, we use this to define $p$-adic $L$-functions in families for $\GSp_4 \times \GL_2$ by pairing a family of $H^2$ eigenclasses for $\GSp_4$ with the product of a $\GL_2$ cusp-form family and an auxiliary Eisenstein series. The existence of this $p$-adic $L$-function then allows us to prove a reciprocity law relating the Euler system of \cite{HJS20} to critical complex $L$-values, and thus to prove the Bloch--Kato conjecture.

 These new methods can also be used to strengthen the results of \cite{LZ20} for the degree 4 motive of $\GSp_4$; for reasons of space, we shall pursue this in a forthcoming paper. Our methods also give, as a by-product, the construction of a ``$p$-adic Eichler--Shimura isomorphism in families'' for $\GSp_4$, interpolating the comparison isomorphisms between de Rham and \'etale cohomology for all (or almost all) specialisations of a $\GSp_4$ Hida family. Our results give an interpolation of the comparison isomorphism after projecting to a specific filtration step of de Rham cohomology, corresponding to $H^2$ of automorphic vector bundles. Our results are thus complementary to the recent work of Diao et al \cite{diao-rosso-wu21} which interpolates the filtration step corresponding to $H^0$.

\begin{remark}
 The switch from ``Hida'' to ``Coleman'' theory allows us to define $p$-adic $L$-functions for finite-slope families, rather than just for ordinary (i.e.~slope 0) families. However, this comes at a price: the use of Coleman theory requires an overconvergence condition on the Eisenstein series, which does not hold for the 2-parameter family of Eisenstein series used in \cite{LPSZ1}. So the price we pay for including the second weight variable $r_2$ is that we lose sight of the cyclotomic variable -- for each automorphic representation $\pi \times \sigma$ of $\GSp_4 \times \GL_2$, there is an interval of integers $n$ such that $V(\pi)^* \otimes V(\sigma)^*(-n)$ is critical, but in the present paper we can only prove the Bloch--Kato conjecture for a specific $n$, corresponding to the lower endpoint of this interval. Even in the ordinary case, to prove the Bloch--Kato conjecture for all of the critical twists, we would need a version of higher Hida (rather than Coleman) theory for $\GSp_4$ with both $r_1$ and $r_2$ varying. Such a theory is not available at present, although analogous results for Hilbert modular groups have been announced by Giada Grossi \cite{grossi21}.
\end{remark}

\emph{Acknowledgements.} We would like to thank George Boxer and Vincent Pilloni for answering our questions about their beautiful theory. We are very grateful for their patience.


\section{Preliminaries}

 Throughout this paper $p$ is a prime.

 \subsection{The group $G$}

  Let $G=\GSp(4)$, with respect to the anti-diagonal Hermitian form with matrix $J = \begin{smatrix} & &&1\\&&1\\&-1\\-1\end{smatrix}$. Write $B_G$ for the Borel subgroup consisting of upper-triangular matrices, and write $P_{\Kl}$ and $P_{\Sieg}$ for the Klingen and Siegel parabolic subgroups containing $B_G$. We then have the Levi decompositions
  \[ B_G = T N_{B}, \qquad P_{\Sieg}= M_{\Sieg} N_{\Sieg},\qquad P_{\Kl}= M_{\Kl} N_{\Kl},\]
  where $T$ is the diagonal torus.

  The Siegel parabolic $P_{\Sieg}$ and its Levi $M_{\Sieg}$ plays a distinguished role in our constructions, since it is conjugate to the centraliser of the cocharacter defining the Shimura datum; the Klingen parabolic is less important here (in contrast with our previous paper \cite{LPSZ1}). Hence we shall often write simply $P_G$, $M_G$ for $P_{\Sieg}, M_{\Sieg}$. We identify $M_{G}$ with $\GL_2 \times \GL_1$ via $\stbt{A}{0}{0}{\star} \mapsto (A, \nu)$, where $\nu$ is the symplectic multiplier.

  Let $W_G=N_G(T)/T$ denote the Weyl group of $(G,T)$. The group $W_G$ is generated by the $T$-cosets of the elements
  $s_1=\begin{smatrix} 1\\ & & 1 \\ & -1 \\ &&&1 \end{smatrix}$
  and
  $s_2=\begin{smatrix}  & 1\\ -1 \\ &&& -1 \\ &&1\end{smatrix}$.

  Let $W_{M_G} = \langle s_2 \rangle $ denote the Weyl group of $(M_G, T)$, and let ${}^{M} W_G=W_{M_G}\backslash W_G$. This has a distinguished set of coset representatives (the \emph{Kostant representatives}) given by
  \[ {}^MW_G=\{ \id, w_1, w_2, w_3 \}\]
  where $w_1 = s_1$, $w_2 = s_1 s_2$, $w_3 = s_1 s_2 s_1$. These have lengths $\ell(w_i) = i$. We use $w_G^{\max}$ for the long Weyl element of $G$, and $w_{M_G}^{\max} = s_2$ the long Weyl element of $M_G$.

  \begin{remark}
   Note that:
   \begin{enumerate}[(i)]
    \item Since $W_G$ permutes the coordinates of the diagonal torus, we can realize it as subgroup of $S_4$. Then ${}^{M}W_G$ identifies with the permutations $w \in W_G$ such that $w(1)<w(2)$.

    \item The map $w \mapsto w_{M}^{\max} \cdot w \cdot w_{G}^{\max}$ preserves ${}^M W_G$, and interchanges $w_i$ with $w_{3-i}$.

    \item The group $\GSp_4$ has the unusual and convenient property that all the elements of ${}^M W_G$ have distinct lengths, so the numbering $w_i$ makes sense.

    \item There is a slight conflict here with the notations of \cite{boxerpilloni20} where $w_0$ is the \emph{long} Weyl element, hence the use of $w_G^{\max}$ here.\qedhere
   \end{enumerate}
  \end{remark}


 \subsection{The group $H$}

  Let $H=\GL_2\times_{\GL_1}\GL_2$, and write $B_H$ for the Borel subgroup of upper-triangular elements of $H$; it has a Levi decomposition $B_H=T_H N_H$. We observe that the Levi subgroup $M_H\subset B_H$ identifies with $T_H$. Hence $W_{M_H}=\{1\}$, so ${}^{M}W_H=W_H \cong W_{\GL_2}\times W_{\GL_2}$.
  We consider $H$ as a subgroup of $G$ via
  \[ \iota: \left[\begin{pmatrix} a & b\\ c & d\end{pmatrix}, \begin{pmatrix} a' & b'\\ c' & d'\end{pmatrix}\right]\mapsto \begin{smatrix} a &&& b\\ & a' & b' & \\ & c' & d' & \\ c &&& d\end{smatrix}.\]
  For future use, we write $H_i$ for the $i$th $\GL_2$-factor of $H$, and we write $T_i$ for the torus of $H_i$ and $\Delta$ for the joint subgroup identified with the character $\nu$, so $T=T_1\times_{\Delta}T_2$. For $i=1,2$, write $\varpi_i:T\rightarrow T_1$ for the natural projection map.


 \subsection{Algebraic weights and roots}

  As in \cite[\S 5.1.1]{pilloni20} and \cite[\S 2.3.1]{LPSZ1}, we identify characters of $T$ with triples of integers $(r_1, r_2; c)$, with $r_1 + r_2 = c \bmod 2$, corresponding to the character
  \[ \diag(st_1, st_2, s t_2^{-1}, s t_1^{-1}) \mapsto t_1^{r_1} t_2^{r_2} s^{c}.\]

  \begin{remark}
   Note that the earlier paper \cite{LSZ17} uses a slightly different notation, but we shall use the above here.
  \end{remark}

  In this notation the simple positive roots with respect to $B_G$ are $(1,-1; 0)$ and $(0, 2; 0)$; and the half-sum is $\rho_G = (2, 1; 0) \in \tfrac{1}{2} X^\bullet(T)$. A weight $(r_1, r_2; c)$ is dominant for $H$ if $r_1, r_2 \ge 0$, dominant for $M_G$ if $r_1 \ge r_2$, and dominant for $G$ if both of these conditions hold.

  The Weyl group acts (on the left) on the group of characters $X^*(T)$ via\footnote{There is also a ``twisted'' action, but we shall write this out explicitly when it arises, rather than defining general notations for it.}
  \[ (w\cdot \lambda)(t)=\lambda(w^{-1}tw). \]
  Explicitly, the generators of the Weyl group of $G$ act by $s_1 \cdot (r_1, r_2; c) = (r_1, -r_2; c)$ and $s_2 \cdot (r_1, r_2; c) = (r_2, r_1; c)$. The long Weyl element acts by $w_G^{\max} \cdot (r_1, r_2; c) = (-r_1, -r_2; c)$; thus $-w_G^{\max} \lambda$ and $\lambda$ coincide on the derived subgroup $\operatorname{Sp}_4$.

  There is an element $\rho_{G, nc}$, which is half the sum of the ``non-compact'' roots, i.e.~those appearing in the unipotent radical $N_{P_{\Sieg}}$. This is given by $\rho_{G, nc} = (\tfrac{3}{2}, \tfrac{3}{2}; 0) = \tfrac{1}{2}(w_M^{\max} \rho_G + \rho_G)$. Note that $\rho_{G,nc}+ \rho_{M_G} = \rho_G$, where $\rho_{M_G} = \left(\tfrac{1}{2}, \tfrac{-1}{2}; 0\right)$ is the half-sum for $M_G$.

  For $H$, all roots are non-compact (the maximal compact in $H(\RR)$ is abelian), so $\rho_{M_H} = 0$ and $\rho_{H, nc} = \rho_H = (1, 1; 0)$.


 \subsection{Flag varieties and Shimura varieties}

  Write $\FL_G=P_G \backslash G$ for the Siegel flag variety of $G$.
  For a neat open compact subgroup $K\subset G(\Af)$, write $S_{G,K}$ for the canonical model over $\Spec(\QQ)$ of the Shimura variety $G(\QQ)\backslash \cH_2\times G(\Af)/K$; here $\cH_2$ is the Siegel upper half space.
  Depending on the choice of a projective cone decomposition $\Sigma$, let $S_{G,K}^{\tor}\rightarrow \Spec(\QQ)$ be the toroidal compactification of $S_{G,K}$ corresponding to $\Sigma$. Denote by $A$ the universal semi-abelian scheme over $S_{G,K}^{\tor}$.

  We similarly define the flag variety $\FL_H=B_H\backslash H$, the canonical rational model of the Shimura variety $S_{H,L}$ for $L=L^pL_p\subseteq H(\Af)$ an open compact subgroup, and $S^{\tor}_{H,L}$ for its toroidal compactification. The universal semi-abelian scheme over $S^{\tor}_{H,L}$ is given by $\cE_1\boxtimes \cE_2$, where $\cE_i$ is the generalized elliptic curve over the modular curve corresponding to the $i$th $\GL_2$-factor of $H$.

  \begin{notation}
   We shall abuse notation slightly by writing $S_{G,K}^{\tor}$ where $K$ is a subgroup of $G(\Qp)$. By this, we mean $S_{G, K^p K}^{\tor}$ where $K^p$ is some fixed choice of open compact away from $p$. When discussing Shimura varieties for $G$ and $H$ together we shall suppose these tame levels $K^{G, p}$ and $K^{H, p}$ are chosen such that $K^{G, p} \cap H = K^{H, p}$.
  \end{notation}

 \subsection{Coefficient sheaves}

  \subsubsection{Reminders on vector bundles, Weyl chambers, coherent cohomology}

   We briefly recall the conventions for automorphic vector bundles from \bp{\S 4.1.1} and attempt to reconcile them with our earlier work elsewhere.
   \newcommand{\Rep}{\operatorname{Rep}}

   We have a functor $\Rep(M_G) \to VB(S^{G,\tor}_K)$ defined using the torsor $M^G_{\dR}$, and we let $\cV_{\kappa}$, for $\kappa \in X^\bullet(T)$ that is $M_G$-dominant, be the image of the representation of highest weight $\kappa$ (with respect to $B_G \cap M_G$). With these conventions:

   \begin{itemize}
    \item The weight $-2\rho_{nc} = (-3, -3; 0)$ maps to $\Omega^3_{S^{G,\tor}_K}(\log D)$ (see \bp{\S 4.2.3}.

    \item The vector bundles having cohomology only in degree 0 are the $\cV_{\kappa}$ with $\kappa = (k_1, k_2; m)$, $-3 \ge k_1 \ge k_2$ (sic).

    \item If $\kappa = (-, -; m)$ for $m \in \ZZ$, and $\pi$ contributes to the cohomology of $\cV_{\kappa}$, then the $\infty$-type of $\pi$ is $\|\cdot\|^{-m}$ on the centre, so $\diag(\varpi_\ell, \dots, \varpi_\ell) \in Z_G(\QQ_\ell)$ acts on $R\Gamma(S^{G,\tor}_K, \cV_{\kappa})$ as $\ell^m$ times a finite-order character.
   \end{itemize}

   We now recall how these are related to automorphic representations.

   \begin{itemize}

    \item For each $M_G$-dominant $\kappa$, there exists a unique $\nu \in X^\bullet(T)$ such that $\nu + \rho$ is dominant and
    \begin{align*}
     \kappa = -w_{M_G}^{\max} w (\nu + \rho) - \rho &\text{ for some $w \in {}^M W_G$},\\
     \text{or equivalently }\kappa = -w w_G^{\max} (\nu + \rho) - \rho &\text{ for some $w \in {}^M W_G$}.
    \end{align*}
    Then the automorphic representations contributing to the cohomology of $\cV_{\kappa}$ have infinitesimal character $\nu + \rho$.

    \item The character $\nu + \rho$ is dominant, but possibly not integral; and $\nu$ is integral, but possibly not dominant. The cases where $\nu$ is actually dominant integral correspond to cohomological $\infty$-types: the infinitesimal character $\nu + \rho$ is the same as that of the algebraic representation $W_\nu$.

    \item Given $\kappa$, the element $\nu$ is uniquely determined but the element $w$ is not. We let
    \begin{align*}
     C(\kappa)^+ &\coloneqq \{ w \in {}^M W_G: \kappa = -w_{M_G}^{\max} w (\nu + \rho) - \rho\},\\
     C(\kappa)^- &\coloneqq \{ w \in {}^M W_G: \kappa = -w w_G^{\max}(\nu + \rho) - \rho\}.
    \end{align*}
    The involution $w \mapsto w_{M}^{\max} w w_{G}^{\max}$ interchanges $C(\kappa)^+$ and $C(\kappa)^-$.

    \item The definition of $C(\kappa)^{\pm}$ can be rewritten in terms of the dual weights $\nu^\vee = -w_G^{\max} \nu$ and $\kappa^\vee = -w_M^{\max} \kappa$; then it takes the form
    \begin{align*}
     C(\kappa)^+ &\coloneqq \{ w \in {}^M W_G: \kappa^\vee = w(\nu + \rho) - \rho \text{ with $\nu + \rho$ dominant}\},\\
     C(\kappa)^- &\coloneqq \{ w \in {}^M W_G: \kappa = w(\nu^\vee + \rho) - \rho \text{ with $\nu^\vee + \rho$ dominant}\}.
    \end{align*}
    This shows, in particular, that $C(\kappa^\vee)^\pm = C(\kappa)^{\mp}$ (corresponding to Serre duality).

    \item If $\pi$ has infinitesimal character $\nu+\rho$, then $\pi$ contributes to cohomology of $\cV_{\kappa}$ in degree $i$ if $C(\kappa)^+$ contains an element of length $i$, or equivalently $C(\kappa)^-$ contains an element of length $3-i$.

   \end{itemize}

   \begin{remark}
    We have overloaded our notations a little for ``dual weights'', since the notation $\sigma^\vee$ for a weight $\sigma$ will mean $-w_{?}^{\max} \sigma$ where $?$ variously denotes $G$, $M_G$, or sometimes $H$ or $M_H$. However, it will (we hope) always be clear from the context which is intended.
   \end{remark}

  \subsubsection{The BGG complex}

   Given a weight $\nu = (r_1, r_2; c)$ as above, with $r_1 \ge r_2 \ge -1$, we shall define
   \[ \kappa_i(\nu) = w_i(\nu + \rho) - \rho, \qquad 0 \le i \le 3. \]
   These are the weights $\kappa$ such that representations of infinitesimal character $\nu^\vee + \rho$ contribute to $R\Gamma(S^{G,\tor}_K, \cV_{\kappa})$. They are explicitly given by
   \begin{equation}
    \label{eq:ourweights}
    \begin{aligned}
     \kappa_0(\nu) &= (r_1, r_2; c) &  \kappa_1(\nu) &= (r_1, -r_2 - 2; c)\\
     \kappa_2(\nu) &= (r_2 - 1, -r_1 -3; c) & \kappa_3(\nu) &= (-r_2-3, -r_1-3; c).
    \end{aligned}
   \end{equation}
   Often we shall take $c = (r_1 + r_2)$; the sheaves $\cV_{\kappa_i(\nu)}$, where $\nu$ has the form $(r_1, r_2; r_1+r_2)$, will be the main sheaves of interest in this work. Note that $\cV_{\kappa_i(\nu)}$ has cuspidal cohomology in degree $3 - i$, and the centre acts on this cohomology via $\diag(p,p,p,p) \mapsto p^{r_1 + r_2}$ up to a finite-order character.

   If $r_1 \ge r_2 \ge 0$, so $\nu$ is dominant, then the sheaves $\cV_{\kappa_i(\nu)}$ are distinct, and they are the terms in the BGG complex quasi-isomorphic to $\cW_{\nu} \otimes \Omega^\bullet(\log D)$.

   \begin{note}
    If $\nu$ is dominant, then $C(\kappa_i(\nu))^- = \{ w_i\}$, and dually $C(\kappa_i(\nu))^+ = \{w_{3-i}\}$. If $\nu = 0$, then $\cV_{\kappa_i(\nu)} = \Omega^i(\log D)$, so the BGG complex is simply the de Rham complex.
   \end{note}

  \subsubsection{Comparison with other works}

   \begin{itemize}
   \item Comparison to \cite{LPSZ1}: the (Hecke-equivariant) vector bundle we attached to $(k_1, k_2; m)$ in \cite{LPSZ1} is now associated to $(-k_2, -k_1; m)$. That is, the conventions differ by $w_{0, M}\cdot  w_{0, G}$, and correcting for this, the collection of sheaves \eqref{eq:ourweights} agrees with \S 5.2 of \emph{op.cit.}.

   \item Comparison to \cite{boxerpilloni20}:
    Our notations are consistent with the general theory described in \cite{boxerpilloni20} but \textbf{not} with the specific choices of conventions made in the application to symplectic groups in \S 5.14 of \emph{op.cit.}. More precisely, we take the upper-triangular Borel subgroup throughout; whereas in \emph{op.cit.} the Borel subgroup consists of matrices of the form
    \( \begin{smatrix}
     \star & \star &  &  \\
      & \star &  &  \\
     \star & \star & \star & \star\\
     \star & \star &  & \star
    \end{smatrix}\),
    so $B_G \cap M_G$ is the same in both conventions, but the roles of the parabolics $P_\mu$ and its opposite $P_{\mu}^{\operatorname{std}}$ are swapped.

    Another discrepancy is that in \emph{op.cit.} the weights considered all have the form $(k_1, k_2; -k_1-k_2)$. This is (to some extent) cancelled out by the fact that in \S 5.14 of \emph{op.cit.} the Hecke operators considered are $\diag(1, 1, p^{-1},p^{-1})$ etc.
   \end{itemize}

   \begin{remark}
    One can choose to have weights with non-negative parameters (up to a constant shift) indexing the weights of automorphic vector bundles with nontrivial $H^0$, or indexing the dominant weights for $G$; but one can't have both. We have chosen the latter, while \bp{\S 5.14} chooses the former.
   \end{remark}

 \subsection{Hecke operators}
  \label{sect:heckeops}

  Let $\Iw_G(p)$ denote the Iwahori subgroup $\{ g \in G(\Zp): g \pmod{p} \in B_G\}$. We shall consider the following operators in the Hecke algebra of level $\Iw_G(p)$, acting on the cohomology of any of the sheaves \eqref{eq:ourweights}:
  \begin{align*}
   \cU_{\Sieg} &= [\diag(p, p, 1, 1)], & \cU'_{\Sieg} &= [\diag(1, 1, p, p)]\\
   \cU_{\Kl} &= p^{-r_2} \cdot [\diag(p^2, p, p, 1)], & \cU'_{\Kl} &= p^{-r_2} \cdot [\diag(1, p, p, p^2)]\\
   \cU_B &= \cU_{\Sieg} \cdot \cU_{\Kl}, & \cU'_B &= \cU'_{\Sieg} \cdot \cU'_{\Kl},\\
   \langle p \rangle &= p^{-(r_1 + r_2)} \left[\diag(p, \dots, p)\right].
  \end{align*}
  The factor of $p^{-r_2}$ included with $\cU_{\Kl}$ and $\cU'_{\Kl}$, and the factor of $p^0$ for the Siegel operators, implies that all four operators are ``optimally integrally normalised'' (i.e.~the Hecke operators preserve a $\ZZ_{(p)}$-lattice, and the power of $p$ is minimal with this property).

 \subsection{Automorphic representations}

  Let $\pi$ be a cuspidal automorphic representation of $G$. We say $\pi$ has ``weight $(r_1, r_2)$'' , for integers $r_1 \ge r_2 \ge -1$, if its infinitesimal character is $\nu^\vee + \rho$, where $\nu = (r_1, r_2; r_1 + r_2)$.
  \[ \omega_{\pi} = |\cdot|^{-(r_1 + r_2)} \widehat{\chi}\]
  for some Dirichlet character $\chi$, where $\widehat{\chi}$ is the adelic character mapping a uniformizer at $\ell$ to $\chi(\ell)$ for almost all primes $\ell$, as in \cite[\S 2.2]{LPSZ1}. For brevity, we say that ``$\pi$ has weight $(r_1, r_2)$ and character $\chi$''.

  We shall also suppose that $\pi$ is globally generic. Hence it is \emph{quasi-paramodular} in the sense of \cite{okazaki}; that is, there exists an explicit subgroup $K(\pi) = K(N_\pi,M_\pi) \subset G(\Af)$, the quasi-paramodular subgroup (depending on the conductor $N_\pi$ of $\pi$, and the conductor $M_\pi$ of its central character), such that $\pif$ has one-dimensional invariants under $K(\pi)$. (See also \cite{robertsschmidt07} for the case of trivial central characters.)

  \begin{definition}
   Let $S$ be a finite set of primes including all primes such that $\pi_\ell$ is ramified, and let $\mathbb{T}^S$ be the Hecke algebra $E[G(\Af^S) / G(\widehat{\ZZ}^S)]$. Then $\pif$ determines a ring homomorphism $\lambda_\pi^S : \mathbb{T}^S \to E$; we write $I^S_\pi$ for its kernel.
  \end{definition}

  We shall consider the localisation at $I^S_{\pi}$ of various finite-dimensional $E$-vector spaces with $\mathbb{T}^S$-actions; this can be concretely defined as the maximal $E$-subspace on which the operators $t - \lambda^S_{\pi}(t)$ are nilpotent for all $t \in \mathbb{T}^S$ (a ``generalised eigenspace''). As with all localisations, this is an exact functor, while the usual eigenspace is not.

  \begin{proposition}
   Let $n = 1$ or $n = 2$. Then the localised cohomology groups $H^i(S^{\tor}_{K(\pi)}, \cV_{\kappa_n(\nu)})_{(I_\pi^S)}$ and $H^i(S^{\tor}_{K(\pi), E}, \cV_{\kappa_n(\nu)}(-D))_{(I_\pi^S)}$ are zero for $i \ne 3-n$; and for $i = 3-n$, both are 1-dimensional and the natural map between them is an isomorphism. In particular, these localisations are independent of $S$.
  \end{proposition}

  \begin{proof}
   By results of Su \cite{su-preprint} (building on earlier works of Harris and others), the coherent cohomology of $S^{\tor}_{K(\pi)}$ can be expressed as the $(\mathfrak{p}, K)$-cohomology of the space of automorphic forms of level $K(\pi)$. Hence it has a filtration, stable under the Hecke action, whose graded pieces are the $K(\pi)$-invariants in the finite parts of automorphic representations of $G$.

   The only representations which can contribute to the localisation at $I_\pi^S$ are those which are locally isomorphic to $\pi$ at all places outside $S$. By Arthur's classification, we can conclude that $\sigma_v$ is non-generic for some place $v$, but lies in the same $L$-packet as the generic representation $\pi_v$. The non-generic, holomorphic representation in the $L$-packet of $\pi_\infty$ does not contribute to $H^1$ or $H^2$. For a finite place $v$, by \cite[Theorem 6.10]{okazaki}\footnote{Or more precisely its proof: we do not know that $\sigma_v$ is tempered, but the alternate input that it is a non-generic member of a generic $L$-packet implies that it lands in one of the same Sally--Tadic types considered in op.cit.}, $\sigma_v$ is not quasi-paramodular (of any level) so it cannot contribute to the cohomology of level $K(\pi)$. So we are left with the contribution from $\pi$ itself; and since $\pi$ has multiplicity one in the discrete spectrum, and the $K(\pi)$-invariants of $\pif$ are 1-dimensional, we are done.
  \end{proof}

  \begin{definition}
   Let $\cW(\pif)_E$ be the $E$-rational part of the Whittaker model of $\pif$, as in \cite[\S 10.2]{LPSZ1}, and define
   \[ S^2(\pi, E) = \Hom_{E[G(\Af)]}\left( \cW(\pi)_E, \varinjlim_K H^2(S^{\tor}_{K, E}, \cV_{\kappa_1(\nu)}(-D))\right), \]
   which is a 1-dimensional $E$-vector space. We define $S^2(\pi, F)$ for any extension $F/E$ similarly.
  \end{definition}

  Since the space of $K(\pi)$-invariants in $\cW(\pif)_E$ has a canonical basis vector $W^{\new}_\pi$, normalised so that $W^{\new}_\pi(1) = 1$, we can identify $S^2(\pi, E)$ with $H^i(S^{\tor}_{K(\pi), E}, \cV_{\kappa_n(\nu)}(-D))_{(I_\pi^S)}$ via evaluation at the new vector. Given $\eta \in S^2(\pi, E)$, we let $\eta_{\sph}$ be its image under this map.

  \begin{definition}
   Given a non-zero $\eta \in S^2(\pi, L)$, where $L$ is some $p$-adic field with an embedding from $E$, we define periods $\Omega_p(\pi,\eta) \in L^\times$ and $\Omega_\infty(\pi, \eta) \in \CC^\times$ as in \cite[\S 6.8]{LPSZ1}.
  \end{definition}

  These two periods are only unique up to multiplication by $E^\times$, but the ratio $\Omega_p / \Omega_\infty$ is uniquely determined once $\eta$ is given.

 \subsection{P-stabilisation}\label{sect:pstab}

  \begin{definition}
   Suppose $\pi$ is unramified at $p$. By a \emph{$p$-stabilisation} of $\pi$, we mean a choice of one among the $W_G$-orbit of characters of $T(\Qp)$ from which $\pi_p$ is induced.
  \end{definition}

  Extending $E$ if necessary, we suppose that the $p$-stabilisations take values in $E$. They biject with the orderings of the Hecke parameters $\alpha, \beta, \gamma, \delta$ of $\pi_p$ respecting the relation $\alpha\delta = \beta\gamma = p^{r_1 + r_2 + 3} \chi_\pi(p)$, and also with the systems of eigenvalues of $\cU'_{\Sieg}$ and $\cU'_{\Kl}$ acting on $(\pi_p)^{\Iw(p)}$:  we can order the parameters so that $\cU'_{\Sieg}$ acts as $\alpha$, and $\cU_{\Kl}'$ acts as $\tfrac{\alpha\beta}{p^{r_2 + 1}}$.

  \begin{definition}
   Let $\mathbb{T}^{-}_{\Iw}$ be the product of $\mathbb{T}^{S \cup \{p\}}$ and the subalgebra of $E[\Iw(p) \backslash G(\Qp) /\Iw(p)]$ generated by $\cU'_{\Sieg}$, $\cU'_{\Kl}$, and $\langle p \rangle$, so that a $p$-stabilisation of $\pi$ determines a character $\lambda^-_\pi: \mathbb{T}^- \to E$ with some kernel $I_\pi^-$.
  \end{definition}

  We say that the $p$-stabilisation is \emph{$p$-regular} if its stabiliser in the Weyl group is trivial; if this holds, the generalised eigenspace in $(\pi_p)^{\Iw(p)}$ associated to $\lambda^-_\pi$ is 1-dimensional. We say a $p$-stablisation is \emph{ordinary} if it maps $\cU'_{\Sieg}$ and $\cU'_{\Kl}$ to $p$-adic units (with respect to some embedding of $E$ into a $p$-adic field $L$). If $\pi$ has regular weight, an ordinary $p$-stabilisation is unique if it exists, and if so, it is automatically $p$-regular; this is no longer true if $r_2 = -1$.

  Any choice of $p$-regular $p$-stabilisation determines an isomorphism
  \[ S^2(\pi, E) \cong H^2\left(S^{\tor}_{K^p(\pi) \Iw(p), E}, \cV_{\kappa_n(\nu)}(-D) \right)_{(I_\pi^-)},\qquad \eta \mapsto \eta_{\Iw}, \]
  given by evaluation at the vector $W_{\alpha, \beta}^{\prime, \Iw} \in \cW(\pi_p)$ defined in \cite{LZ20b-regulator}.

 \section{Summary of higher Coleman theory for $G$}

 \subsection{Outline}

  We briefly survey the results we shall need from \cite{boxerpilloni20}. Here the group $H$ will not appear, so we drop unnecessary subscripts $G$; we also fix a level $K^p$ away from $p$ (and suppress it from the notation).

  \begin{itemize}
   \item We begin with the \textbf{classical cohomology} $R\Gamma(\kappa_n) = R\Gamma(S_{K^p\Iw(p)}^{\tor}, \cV_{\kappa_n}) \otimes_{\QQ} \Qp$, a complex of finite-dimensional $\Qp$-vector spaces, and its cuspidal analogue $R\Gamma(\kappa_n, \cusp)$.

   \item For each $w \in {}^M W$, each algebraic weight $\kappa$ dominant for $M_G$, and each sign $\pm$, we have complexes of $\Qp$-Banach spaces $R\Gamma_w(\kappa)^{\pm}$ \textbf{(``overconvergent cohomology'')}. These are defined as relative cohomology groups for a stratification of $\cS_{K^p\Iw(p)}^{\tor}$, with coefficients in $\cV_{\kappa}$. We have finite-slope decompositions for the action of $(\cU_{\Sieg}, \cU_{\Kl})$ on $R\Gamma_w(\kappa)^+$, and for $(\cU'_{\Sieg}, \cU'_{\Kl})$ on $R\Gamma_w(\kappa)^-$.

   \item There is a spectral sequence (for either sign) relating the $R\Gamma_w(\kappa)^{\pm}$ for varying $w$ to the classical cohomology. If $\kappa$ is regular, and we localise at an eigenspace of strictly small slope, then only one $w \in {}^M W$ contributes to this spectral sequence (the unique one such that $w \in C(\kappa)^{\pm}$) so the localised spectral sequence degenerates.

   \item There are complexes of $\Qp$-Banach spaces $R\Gamma_{w, \an}(\nu)^{\pm}$ (\textbf{``locally analytic cohomology''}), which are defined by replacing the coherent sheaves $\cV_{\kappa}$ with some Banach sheaves of analytic functions (for sign $+$) or distributions (for sign $-$). These are defined for any locally analytic character $\nu$.

   \item For $\kappa$ algebraic and $M$-dominant, there is a second spectral sequence relating $R\Gamma_w(\kappa)^{\pm}$ to the locally-analytic cohomologies $R\Gamma_{w, \an}(\nu)^{\pm}$ for a collection of $\nu$ depending on $\kappa$ and $w$. Again, for regular weights this spectral sequence will degenerate on the small-slope part.

   \item The locally-analytic cohomology complexes also make sense with coefficients in an affinoid algebra $A$, giving complexes of Banach $A$-modules $R\Gamma_{w, \an}(\nu_A)^{\pm}$ (``\textbf{cohomology in families}''), and there is a Tor spectral sequence relating these to $R\Gamma_{w, \an}(\nu)^{\pm}$ when $\nu$ is a specialisation of $\nu_A$ at some point of $\operatorname{Max}(A)$.
  \end{itemize}

  \begin{remark}[Levels and overconvergence radii]
   It is important to note that the complexes $R\Gamma_w(\kappa)^{\pm}$ and $R\Gamma_{w, \an}(\nu_A)^{\pm}$ are not quite canonical, since they depend on various choices (radii of overconvergence, and levels at $p$). However, the maps arising from changing these parameters induce isomorphisms on the finite-slope parts.
  \end{remark}

  All of the above constructions also have a ``cuspidal'' flavour (tensoring all the coefficient sheaves with the ideal sheaf of the toroidal boundary).
%

 \subsection{The spectral sequences} We now spell out the above constructions in slightly more detail.

  \begin{remark}
   In \emph{op.cit.} the tame level $K^p$ is assumed to be \emph{neat}, but it will be convenient to relax this: if $K^p$ is any open compact in $G(\Af^p)$, then we choose an auxiliary smaller subgroup $L^p \trianglelefteqslant K^p$ which is neat (such subgroups always exist) and define complexes $R\Gamma_w(K^p, \kappa)^\pm$ etc as the $(K_p / L_p)$-invariants of the corresponding complexes at level $L_p$. This commutes with taking cohomology, since $K_p/L_p$ is a finite group and all our complexes will be complexes of $\QQ$-vector spaces. (This will allow us to choose $K^p$ to be a paramodular group, in the sense of \cite{robertsschmidt07}.)
  \end{remark}

  \subsubsection{Overconvergent to classical cohomology}

   This is the spectral sequence of \bp{Theorem 5.15}. In the ``minus'' case (which interests us most) it will be
   \begin{equation}
    \label{eq:bruhatSS}
    E_1^{ij} = H^{i + j}_{w_{3-i}}(\kappa, \cusp)^{-, \fs} \Rightarrow H^{i+j}(\kappa, \cusp)^{-,\fs},
   \end{equation}
   and similarly for non-cuspidal cohomology. If we apply the strictly-small-slope condition $(-, \sss^M)$, then only the terms with $w \in C(\kappa)^-$ survive; if $\kappa = \kappa_i(\nu)$ with $\nu$ dominant, then this is only $w = w_i$, so we have (\bp{Theorem 5.66}):
   \[ R\Gamma(\kappa_i(\nu), \cusp)^{-,\sss^M} = R\Gamma_{w_i}(\kappa_i(\nu), \cusp)^{-, \sss^M}. \]
   The same applies without the cuspidal condition, but for cuspidal cohomology we obtain the additional information that the complex is concentrated in degrees $[0, 3-i]$, by \bp{Theorem 5.15}.

  \subsubsection{Locally-analytic to overconvergent}

   We only really care about the edge map of the spectral sequence here. In our notation this map is
   \[
    R\Gamma_{w_1, \an}(\nu, \cusp)^{-, \fs} \to R\Gamma_{w_1}(\kappa_1(\nu), \cusp)^{-,\fs}.
   \]
   By \bp{Corollary 6.36}, this is an isomorphism on the eigenspaces satisfying the slope condition $(-,\sss_{M, w_1}(\kappa_1))$ in \emph{op.cit.}. Moreover, $R\Gamma_{w_1, \an}(\nu, \cusp)^{-, \fs}$ is concentrated in degrees $[0,1, 2]$, by \bp{Theorem 6.29}.

  \subsubsection{The Tor sequence}

   The last spectral sequence we need takes the form
   \[ E_2^{pq} = \Tor_{-p}^A( H^q_{w_1, \an}(\nu_A, \cusp)^{-, \fs}, \nu) \Longrightarrow H^{p+q}_{w_1, \an}(\nu, \cusp)^{-, \fs}\]
   where $A$ is some affinoid algebra with a continuous character $\nu_A: T(\Zp) \to A^\times$, and the algebraic character $\nu$ defines a $K$-point of $\operatorname{Max}(A)$. Here we cannot rely on ``small slope'' arguments in order to make the sequence degenerate (since there is no reason to expect the slopes on the higher Tor terms to be bigger than the slopes on the $\Tor_0$ term).

 \subsection{Slope conditions}

  We consider the Hecke operators normalised as in \cref{sect:heckeops}. on the cohomology of the sheaves $\cV_{\kappa}$, for $\kappa$ as in \eqref{eq:ourweights}. Thus each operator is ``minimally integrally normalised'' acting on the classical cohomology (cuspidal or non-cuspidal), i.e.~its slopes are $\ge 0$.

  \subsubsection{Expected and provable slope bounds for overconvergent cohomology}

   \bp{Conjecture 5.29} predicts lower bounds for the slopes of the Hecke operators acting on the overconvergent cohomology complexes $R\Gamma_w(K^p, \kappa)^{\pm}$ and $R\Gamma_w(K^p, \kappa, \cusp)^{\pm}$; there is a similar conjecture \bp{Conjecture 6.33} for the locally-analytic cohomology complexes (where there are more possibilities, since these are defined for weights which might not be $M_G$-dominant).

   These are summarized by the following table, in which we compute for various elements $w \in W_G$ the character $w^{-1}(\kappa + \rho) - \rho$, and how it pairs with the anti-dominant cocharacters $\diag(1,1,x,x)$ and $\diag(1,x,x,x^2)$ defining the operators $\cU'_{\Sieg}$ and $\cU'_{\Kl}$. We take $\kappa = \kappa_1 = (r_1, -r_2-2; r_1 + r_2)$, and subtract $r_2$ from all entries in the bottom row (since this is our normalising constant for $\cU'_{\Kl}$). This gives the following table:

   \begin{center}
   \begin{tabular}{c|cccc|c}
    $w=$  & id & $(w_1)$ & $w_2$ & $w_3$ & $w^{\max}_M w_1$\\
    \hline
    $\cU'_{\Sieg}$ & $r_2+1$ & $(0)$ & $0$         & $r_1+2$     & $r_1+r_2+3$\\
    $\cU'_{\Kl}$   & $0$     & $(0)$ & $r_1-r_2+1$ & $r_1-r_2+1$ & $r_1+r_2+3$\\
   \end{tabular}
  \end{center}

  We do not know this conjecture in full, but we know a weaker statement, \bp{Theorem 5.33 and Theorem 6.35}, in which we replace $w^{-1}(\kappa_1 + \rho) - \rho$ with $w^{-1} \kappa_1$. This gives the following bounds:

  \begin{center}\begin{tabular}{c|cccc|c}
  $w=$ & id & $(w_1)$ & $w_2$ & $w_3$ & $w^{\max}_M w_1$\\
   \hline
   $\cU'_{\Sieg}$ & $r_2+1$ & $(-1)$ & $-1$        & $r_1-1$ & $r_1 + r_2 + 1$\\
   $\cU'_{\Kl}$   & $0$     & $(0) $ & $r_1-r_2-2$ & $r_1-r_2-2$& $r_1 + r_2 + 2$\\
  \end{tabular}\end{center}

  \begin{remark}
   More precisely: for each $w$ we consider the weight $w^{-1}(\kappa_1 + \rho) - \rho$, and how it pairs with the anti-dominant torus elements $(1, 1, p, p)$ and $(1, p, p, p^2)$. The bottom row gets multiplied by $p^{-r_2}$ from our normalisation of Hecke operators.

   In the second table, we do the same computation with $w^{-1}\kappa_1$ instead of $w^{-1}(\kappa_1 + \rho) - \rho$.
  \end{remark}

%

   \subsubsection{Small slope conditions} We can now compute which eigensystems satisfy the various small-slope conditions of \bp{\S 5.11}.

   \begin{proposition}
    For the weight $\kappa_1 = (r_1, -r_2-2; r_1 + r_2)$, with $r_1 \ge r_2 \ge 0$, we have the following:

    \begin{itemize}
     \item The ``small slope'' condition $(-,\ss^M(\kappa_1))$ is
     \[ \lambda(\cU'_{\Sieg}) < 1 + r_2, \qquad \lambda(\cU'_{\Kl}) < 1 + r_1 - r_2.\]
     \item The ``strictly small slope'' condition $(-, \sss^M(\kappa_1))$ is
     \[ \lambda(\cU'_{\Sieg}) < 1 + r_2, \qquad \lambda(\cU'_{\Kl}) < -2 + r_1 - r_2.\]
     \item The condition $(-, \sss_{M, w_1}(\kappa_1))$ is implied by $(-,\sss^M(\kappa_1))$, and similarly for the non-strict versions.
    \end{itemize}
   \end{proposition}

   \begin{proof}
    This follows from the tables of the previous section.
   \end{proof}

   Thus, in any regular weight, ordinary classes have small slope (as one might reasonably expect). However, they fail the strict-small-slope condition unless $r_1 - r_2 \ge 3$.

 \subsection{Families of eigenclasses}
  \label{sect:families}

  We now briefly indicate the consequences of this theory for a cohomological automorphic representation. We let $\pi$ be as in \cref{sect:pstab}, and let $K^p$ denote the paramodular subgroup away from $p$ of the appropriate level. We suppose $\pi$ is ordinary at $p$, and also that $r_2 \ge 0$ and $r_1 - r_2 \ge 3$. Let $\lambda_\pi^-$ be its ordinary $p$-stabilisation, and $I_\pi^-$ the kernel of $\lambda^-_{\pi}$.

  Then the classicity theorems of higher Coleman theory recalled above give quasi-isomorphisms
  \[
   R\Gamma_{w_1, \an}(\nu, \cusp)^{-, \sss} \cong R\Gamma_{w_1}(\kappa_1, \cusp)^{-,\sss} \cong R\Gamma(\kappa_1, \cusp)^{-,\sss},
  \]
  and since $\lambda_{\pi}^-$ is ordinary, the localisation of $R\Gamma(\kappa_1, \cusp)$ at $I_{\pi}^-$ is contained in the strictly-small-slope part. So, from the results quoted above, the localisation of these three complexes at $I_\pi^-$ is 1-dimensional in the $H^2$, and vanishes in all other degrees.

  For the ``family'' cohomology, we use the local criterion for flatness: for a finitely-generated module $M$ over a local ring $(A, I)$, if $\Tor_1^A(M, A/I) = 0$, then $M$ is free. From this and the above $\Tor$ spectral sequence, we conclude that the localisation of $H^i_{w_1, \an}(\nu_A^\vee, \cusp)^{-, \fs}$ at $I_\pi^-$ is zero for $i \ne 2$, and is locally free of rank 1 over $A$ for $i = 2$.

  Thus the $p$-stabilised new-vector $\eta_{\Iw} \in H^2(K^p, \kappa_1, \cusp)_{(I_\pi)}$ associated to $\pi$ deforms to an analytic family over some open affinoid in the 2-dimensional weight space $\cW \times \cW$: for some sufficiently small $A \ni (r_1, r_2)$, we obtain a class $\uet \in H^2_{w_1, \an}(\nu_A^\vee, \cusp)$, and a homomorphism $\underline{\lambda}: \mathbb{T}^- \to A$ lifting $\lambda_{\pi}$, such that $ \mathbb{T}^-$ acts on $\uet$ via $\underline{\lambda}$, and the specialisation of $\uet$ at $(r_1, r_2)$ is $\eta_{\Iw}$. These results will be used below to construct our $p$-adic $L$-functions.

  Similar arguments apply to the modules $H^{3-i}_{w_i, \an}(\nu_A^\vee, \cusp)$ for all $i$ (although in the case $i = 0, 3$ the classical eigenspace associated to $\pi$ may be zero, if $\pi$ is a Yoshida lift).

  \begin{remark}
   Unsurprisingly, we could relax the assumption that $\pi$ be Borel-ordinary at $p$, as long as it admits some refinement satisfying the $\sss^M$ and $\sss_{M, w_1}$ conditions.
  \end{remark}


\section{Functoriality of higher Coleman theory (overview)}

 The next few sections, which are the main technical context of the present paper, are devoted to constructing maps relating higher Coleman theory spaces for $G$ and for $H$. More precisely, we shall define three maps of complexes
 \begin{subequations}
  \begin{align}
   R\Gamma^G(\kappa_1, \cusp) &\longrightarrow R\Gamma^H(\tau, \cusp) &&\text{(classical)} \\
   R\Gamma^G_{w_1}(\kappa_1, \cusp)^{-, \fs} &\longrightarrow R\Gamma^H_{\id}(\tau, \cusp)^{-} &&\text{(overconvergent)} \\
   R\Gamma^G_{w_1, \an}(\nu_A, \cusp)^{-, \fs} &\longrightarrow R\Gamma^H_{\id, \an}(\tau_A, \cusp)^{-} &&\text{(locally-analytic)}
  \end{align}
 \end{subequations}
 satisfying appropriate compatibilities. In the first two cases, $\tau$ is a weight for $H$ in an appropriate range depending on $(r_1, r_2)$ (see \cref{sect:branchcoeffs}; in the third case, $\nu_A$ and $\tau_A$ are families of weights for $G$ and $H$ with an appropriate relation between them (see \cref{def:compat}).

 It is important to note that these maps do \emph{not} land in the finite-slope part on the right-hand side. In particular, the resulting complexes depend on various auxiliary choices of parameters, and changing these induces an isomorphism on the finite-slope part but not on the whole complex; in order to obtain a uniquely-determined map we shall pass to the inverse limit.

\section{Maps of flag varieties}
 We now delve a little further into the construction of the overconvergent and locally-analytic cohomology complexes, in order to study how they interact with pullback along $H \into G$.

 \subsection{Bruhat cells and tubes} We consider the following $\Zp$-subschemes of $\FL_G$:

  \begin{definition}
   For $w \in {}^M W_G$, we define:
   \begin{itemize}
   \item $C^G_{w} = P_G \backslash P_G w B_G$, a locally closed subscheme;
   \item $X^G_w = \bigcup_{w' \le w} C^G_{w'}$, a closed subscheme (the closure of $C^G_w$);
   \item $Y^G_w = \bigcup_{w' \ge w} C^G_{w'}$, an open subscheme.
   \end{itemize}
  \end{definition}

  We shall mostly be interested in their special fibres $C^G_{w, \Fp}$ etc. Note that the dimension of $C^G_{w}$, or of $X^G_w$, is $\ell(w)$.

  \begin{remark}
   The double coset $P_G g B_G$ for a given $g \in G$ depends only on the span of the rows of the $2 \times 2$ lower-left submatrix of $g$. If this span is zero, $g$ is in the small cell; if it is $(0, \star)$ then we are in the 1-dimensional cell; if it is 1-dimensional but not equal to $(0, \star)$, then we are in the 2-dimensional cell; if it is the whole of $\AA^2$, we are in the big cell.
  \end{remark}

  Similar definitions apply for $H$ in place of $G$, although we shall only use $C^H_{\id}$ in the present work.


 \subsection{Maps of flag varieties}

  Let $\gamma=\begin{smatrix}1 &&&\\1 & 1 &&\\&&\phantom{-}1 &\\ &&-1& 1 \end{smatrix}$, and let $\hat{\gamma}=\gamma w_1 = \begin{smatrix} 1& & & \\ 1 && 1& \\ & -1&& \\ & \phantom{-}1& & 1\end{smatrix}$. Then the map $\hat\iota: \FL_H \to \FL_G$ sending $B_H h$ to $P_G h \hat\gamma$ is a closed immersion of $\Zp$-schemes (since it is a translate of the obvious closed embedding induced by $H \into G$). It maps the identity to $w_1$, since $\gamma \in P_{\Sieg}$.

  \begin{proposition}
   \label{prop:intersectcells}
   The intersections of the Bruhat cells of $\FL_G$ with $\hat\iota(\FL_H)$ are as follows:
   \begin{itemize}
   \item $\hat\iota^{-1}\left(C^G_{\id}\right)= \varnothing$;
   \item $\hat\iota^{-1}\left(C^G_{w_1}\right) = C^H_{\id}$.
   \end{itemize}
  \end{proposition}

  \begin{proof}
   This is an elementary linear-algebra computation. We can identify $\FL_H$ with $\mathbf{P}^1 \times \mathbf{P}^1$ in such a way that the strata are $\infty \times \infty, \infty \times \AA^1$, $\AA^1 \times \infty$ and $\AA^2$. The map to $\FL_G$ sends $((x:y), (X : Y))$ to the plane spanned by $(x,-y,0,-y)$ and $(X,Y,X,0)$. The projection of this plane to the first two coordinate vectors is never 0, so $\hat\iota^{-1}\left(C^G_{\id}\right)= \varnothing$. Its projection to the first coordinate vector is 0 iff $x = X = 0$, i.e.~$\hat\iota^{-1}\left(C^G_{w_1}\right) = C^H_{\id}$.
  \end{proof}

  \begin{remark}
   Conceptually, this is saying that the subspace $\hat\iota(\FL_H) \subset \FL_G$, which is 2-dimensional, intersects the Bruhat strata $C^G_{\id}$ and $C^G_{w_1}$ in the ``expected'' codimensions, and these codimensions are themselves Bruhat strata in $\FL_H$: it has empty intersection with the codimension 3 stratum, and its intersection with the codimension 2 stratum is the unique 0-dimensional. (Its intersection with the codimension 1 stratum $C^G_{w_2}$ also has the expected dimension, but it is not a Bruhat stratum; however, this plays no role in our theory.)
  \end{remark}

 \subsection{Explicit coordinates on the flag variety}
  \label{sect:explicitparam}

  For $w \in {}^M W_G$ we let $U^G_w$ be the ``big cell at $w$'', i.e. the translate of the big cell $P_G \backslash P_G \bar{N}_G$ by $w$. This is accordingly an open neighbourhood of $w$ naturally parametrised by $\bar{N}_G \cong \AA^3$, which we identify with $2 \times 2$ off-symmetric matrices $Z = \stbt x y z x$, via $Z \mapsto P_G \backslash P_G \stbt{1}{}{Z}{1} w$.

  In these coordinates, the action of $g \in G$ is given as follows: we have
  \[ P \backslash (P \stbt{1}{}{Z}{1} w) \cdot g = P \backslash (P \stbt{1}{}{Z}{1} w g w^{-1}) w, \]
  and if we write $w g w^{-1} = \stbt{A}{B}{C}{D}$, then
  \[ (P \stbt{1}{}{Z}{1} w g w^{-1}) = P \stbt{1}{}{Z'}{1},\qquad Z' = (Z B + D)^{-1} (ZA + C).\]

  \begin{remark}
   This is a mildly twisted version of the familiar formula for the action of the Siegel modular group on the complex-analytic Siegel half-space, which is of course nothing but an open subset of the $\CC$-points of $\FL_G$.
  \end{remark}

  As a special case of this formula, if $\delta = \diag(p^3, p^2, p, 1)$ and $w = w_1$, then
  \[ \stbt x y z x \cdot \delta = \stbt{px}{p^{-1}y}{p^3 z}{px}.\]

  Similarly, the big cell for $H$ (at $w = \id$) is isomorphic to $\AA^2$ via $(z_1, z_2)\mapsto B_H \backslash B_H\left( \stbt{1}{}{z_1}{1},\stbt{1}{}{z_2}{1}\right)$. In these coordinates on $U^H_{\id}$ and $U^G_w$, we compute that the map $\hat\iota$ is given by
  \begin{equation}
   \label{eq:gamma-param}
   (z_1, z_2) \mapsto \stbt{z_2}{z_2}{z_1 + z_2}{z_2}.
  \end{equation}

  \begin{remark}
   More conceptually, we can identify the big cells $U^H_{\id}$ and $U^G_{w_1}$ with the Lie algebras $\bar{\mathfrak{n}}_H$ and $\bar{\mathfrak{n}}_G$, and this formula is just the composite of the natural inclusion $\bar{\mathfrak{n}}_H \subset \bar{\mathfrak{n}}_G$ and the adjoint action of $\gamma \in M_G$ on $\bar{\mathfrak{n}}_G$.
  \end{remark}


 \subsection{Analytic geometry: notations}

  Let us write $\tFL^H$ and $\tFL^G$ for the analytic adic spaces (over $\Spa(\Qp, \Zp)$) associated to $\FL_G$ and $\FL_H$. (We use typewriter letters for the flag varieties and open sets in them, since we shall later use calligraphic letters for subsets of the adic Shimura variety.)

  We now describe explicit isomorphisms between certain tubes in $\tFL^G$ and $\tFL^H$ and adic polydiscs. We will need to distinguish between \emph{four} flavours of ``disc'' inside the adic affine line. This is because the ``closed disc'' $\{|.| : |z| \le 1\}$, where $z$ is a coordinate function, is actually open (but not closed) in the adic topology, whereas the ``open disc'' $\{|.| : |z| < 1\}$ is closed (but not open)!

  For $m \in \QQ$, we define
  \[ \cB_m = \{ |.| : |z| \le |p|^{m} \},\qquad \overline{\cB}_m = \bigcap_{m' < m} \cB_m, \]
  and
  \[ \cB^{\circ}_m = \bigcup_{m' > m} \cB_{m'}, \qquad \overline{\cB}^\circ_m = \{ |.| : |z| < |p|^m\}.
  \]
  Thus $\cB^{\circ}_m \subset \overline{\cB}^\circ_{m} \subset \cB_m \subset\overline{\cB}_m$, and (as the notation suggests) $\overline{\cB}_m$ is the closure of $\cB_m$, and similarly $\overline{\cB}_m^\circ$ of $\cB^{\circ}_m$. Moreover, the sets $\overline{\cB}_m - \cB_m$ and $\overline{\cB}^\circ_m -\cB^{\circ}_m$ consist entirely of rank $> 1$ points.

  \begin{remark}
   Compare the four flavours of root subgroups in \bp{\S 3.3.2}. The space $\overline{\cB}_m$ corresponds to the ``dagger affinoid disc'' in Grosse-Kl\"onne's theory of dagger spaces.
  \end{remark}

  More generally, if $A$ is a subset of $\overline{\QQ}_p$, we write $A + \cB_m = \bigcup_{a \in A}(a + \cB_m)$ etc; we shall only use this if $A$ is compact, in which case the union is finite.

 \subsection{Level groups at $p$}

  \begin{notation} Let $t \in \ZZ_{\ge 1}$.
   \begin{itemize}
   \item Let $K_{\Iw}^G(p^t) = \{ g \in G(\Zp): g\bmod p \in B_G(\ZZ/p^t)\}$ be the depth $t$ upper-triangular Iwahori of $G$, and similarly for $H$.
   \item Let $K^H_{\ds}(p^t)$ denote the group $H(\Qp) \cap \hat\gamma K^G_{\Iw}(p^t) \hat\gamma^{-1}$, which is concretely given by
   \[ K^H_{\ds} =
    \left \{ h \in H(\Zp): h = \left(\stbt{x}{y}{0}{z},\stbt{x}{-y}{0}{z}\right) \bmod p^t \text{ \textup{for some} $x,y,z$}\right\}.\]
   \end{itemize}
  \end{notation}

  Note that $K^H_{\ds}(p^t)$ is, fortuitiously, a subgroup of $K^H_{\Iw}(p^t)$.

 \subsection{Tubes of ``radius one''}

  We note that if $\cX$ is the analytic adic space associated to a finite-type $\Zp$-scheme $X$, then there is a specialisation map $\sp: \cX \to X_{\Fp}$ which is a continuous map of topological spaces. If $Z \subset X_{\Fp}$ is a locally closed subset, we let $]Z[$ be the \emph{interior} of $\sp^{-1}(Z)$; this is the adic space corresponding to the tube in the sense of classical rigid geometry, while $\sp^{-1}(Z)$ is not a classical rigid space in general. Of course, if $Z$ is open, then $\sp^{-1}(Z) =\  ]Z[$; on the other hand, if $Z$ is closed, then $\sp^{-1}(Z) =\  \overline{]Z[}$.

  \begin{definition}
   Let $\tU^G_0 =\ ]Y^G_{w_1, \Fp}[$, $\tZ^G_0 = \ \overline{]X^G_{w_1, \Fp}[}$, and $\tI^G_{0,0} = \tZ^G_0 \cap \tU^G_0$.
  \end{definition}

  Note that $\tU^G_0$ is open and $\tZ^G_0$ closed, and both are invariant under the Iwahori $K^G_{\Iw}(p)$ (since the Borel subgroup of $G_{\Fp}$ fixes the mod $p$ Bruhat cells). Thus $\tI^G_{0, 0}$ is a ``partial closure'' of the Bruhat cell $]C_{w_1, \Fp}^G[$.

  We also write $\tZ^H_0 = \overline{]X^H_{\id}[}$ (the preimage of the point $\{\id_H\} \in \FL^H_{\Fp}$) which is stable under $K^H_{\Iw}(p)$, and we formally set $\tU^H_0 = \tFL^H$.

  \begin{proposition}
   We have $\tI^G_{0,0} \subset U^G_{w_1}$, and in the coordinates on $U^G_{w_1}$ described in \cref{sect:explicitparam}, we have
   \[ \tI^G_{0,0} = \{ \stbt x y z x : x, z \in \overline{\cB}_{0}^\circ, y \in \cB_0 \}.\]
   Similarly
   \[ \tI^H_{0,0} = \tZ^H_0 = \{ (z_1, z_2) : z_i \in \overline{\cB}_{0}^\circ \}.\]
  \end{proposition}

  \begin{proof}
   This is an instance of \bp{Lemma 3.21 (5)}.
  \end{proof}

  \begin{proposition}
   \label{prop:Cartadic}
   We have a Cartesian diagram of adic spaces
   \[
    \begin{tikzcd}
     \tZ^H_0 = \tI^H_{0, 0} \dar["\hat\iota"] \rar[hook] &\tU^H_0 \dar["\hat\iota"]\\
     \tI^G_{0,0} \rar[hook] & \tU^G_0
    \end{tikzcd}
   \]
   in which all the morphisms are closed embeddings.
  \end{proposition}

   \begin{proof}
    This follows readily from \cref{prop:intersectcells} and the definition of the $\tU$'s, $\tI$'s and $\tZ$'s.
   \end{proof}

 \subsection{Tubes of smaller radius}

  Let $m, n, t$ be integers with
  \begin{equation}\label{eq:mnt}
   0 \le n \le m < t, \qquad \text{$m > n$ if $n \ne 0$}.
  \end{equation}

  \begin{definition}
   We define subsets $\tI_{m, n}^G \subset \tU_n^G$ in $\tFL^G$ as follows: we let
   \[ \tI_{m, n}^G = \{ \stbt x y z x : x, z \in \overline{\cB}_{m}^\circ, y \in \cB_n + \Zp \}. \]
   (consistently with the $(m, n) = (0, 0)$ case described above). For $n \ge 1$ we set
   \[ \tU_n^G =\{ \stbt x y z x : x, z \in \cB_n^\circ, y \in \cB_n + \Zp \}.\]
   and for $n = 0$ we use the definition above.
  \end{definition}

  \begin{proposition} \
   \begin{enumerate}[(i)]
   \item The sets $\tU_n^G$ and $\tU_n^G - \tI_{m, n}^G$ are open in $\tFL^G$.
   \item If $n > 0$ then $\tI_{m, n}^G = \tI_{m, 0}^G \cap \tU_n^G$.
   \item The sets $\tU_n^G$ and $\tI_{m, n}^G$ are stable under $K_{G, \Iw}(p^t)$, and in the notation of \bp{\S 3.3.3}, we have
   \[ \tU_n^G =\ ]C_{w_1, \Fp}[_{(n, n)}K_{G, \Iw}(p^t),\qquad \tI_{m, n} =\ ]C_{w_1, \Fp}[_{(\overline{m}, n)}K_{\Iw}^G(p^t).\]
   \end{enumerate}
  \end{proposition}

  \begin{proof}
   The first two statements are obvious. For the stability under $K_{G, \Iw}(p^t)$, we treat $n = 0$ and $n > 0$ separately: in the $n = 0$ case, the stability of $\tU_0^G$ is already established, and the stability of $\tI_{m, n}$ follows from the identification $\tI_{m, 0} = \, ]C_{w_1, \Fp}\![\,_{\overline{m}, 0} = \cP\backslash \cP w \cG_{\overline{m}, 0}$ of \bp{\S 3.3.3}.

   For $n > 0$, we argue as in Lemma 3.18 of \emph{op.cit.} to describe the spaces $]C_{w_1, \Fp}[_{(n, n)} K_{\Iw}^G(p^t)$ and $]C_{w_1, \Fp}[_{(\overline{m}, n)}K_{\Iw}^G(p^t)$ as finite unions of translates of $]C_{w_1, \Fp}[_{(n, n)}$, indexed by coset representatives for $N_{B_G}(\ZZ/p^n)$ modulo its intersection with $w^{-1}P_G w$. We can take these coset representatives to be of the form
   \[ \begin{smatrix} 1 \\ &1 & \alpha \\ &&1 \\ &&&1\end{smatrix}, \qquad \alpha \in \ZZ / p^n, \]
   which act as $\stbt x y z x \mapsto \stbt x {y + \alpha} z x$.
  \end{proof}

  \begin{remark}
   One can choose a closed subset $\tZ_m^G$ such that $\tI_{m, n}^G = \tZ_m^G \cap \tU_n^G$; however, it is a little awkward to choose such a subset which is invariant under $K^G_{\Iw}(p^t)$, and in any case the choice of such a subset will not matter for our theory, so we shall not privilege any particular choice.
  \end{remark}

  \begin{remark}
   Formally setting $n = 0$ in the formula defining $\tU_n^G$ for $n > 0$ gives a well-defined and $K^G_{\Iw}(p)$-invariant set (in fact it is exactly $]C_{w_1}[\,$), but this set does not contain $\tI_{0, 0}$. Hence we use the formula of the previous section in which $x, z$ are allowed to ``go off to infinity''.
  \end{remark}

 \subsection{Iwahori-level tubes for $H$}

  For $n < t$ we shall define
  \[ \tU_n^H = \hat\iota^{-1}(\tU_n^G) \subseteq \tFL^H.\]
  This is an open $K_{H, \Iw}(p^t)$-invariant set, containing $\{\id\}$. For $n = 0$ it is the whole of $\tFL^H$.

  \begin{proposition}
   The set $\hat\iota^{-1}(\tI_{m, n}^G)$ is closed in $\tFL_H$ and invariant under $K_{H, \Iw}(p^t)$, and does not depend on $n$; explicitly it is given by
   \[ \tZ_m^H = \{ (z_1, z_2): z_i \in \overline{\cB}^\circ_m\}.\]
  \end{proposition}

  \begin{proof}
   Clear from \eqref{eq:gamma-param}.
  \end{proof}

  We therefore have a Cartesian diagram of closed embeddings generalising \cref{prop:Cartadic} above, for any $m, n, t$ as in \eqref{eq:mnt}:
  \begin{equation}\label{eq:Cartadic2}
   \begin{tikzcd}
    \tZ^H_m \dar["\hat\iota"] \rar[hook] &\tU^H_n \dar["\hat\iota"]\\
    \tI^G_{m,n} \rar[hook] & \tU^G_n.
   \end{tikzcd}
  \end{equation}

  It is convenient to extend the definition of $\tZ_m$ to remove the requirement that $m < t$. For any integers $t \ge 1$ and $m \ge 0$, let us define
  \[ \tZ_m^H \coloneqq \{ (z_1, z_2): z_i \in \overline{\cB}^\circ_m + p^t \Zp\}.\]
  This is invariant under $K_{H, \Iw}(p^t)$, and hence also under $K_{H, \ds}(p^t)$. If $m < t$ then this agrees with the definitions above. We define an open set $\tU_n^H$ similarly. (However, in the $m \ge t$ case we do not know if these sets $\tU_n^H$ and $\tZ_m^H$ can be fitted into a diagram like \eqref{eq:Cartadic2}.)


\section{Pullbacks in overconvergent cohomology}
 \subsection{Adic spaces and period maps}

  We consider the analytifications $\cS_{G,K}^{\an}=(S_K\times \Spec(\Qp))^{\an}$,  $\cS_{G,K}^{\tor}=(S_{G,K}^{\tor}\times\Spec(\Qp))^{\an}$ and $\tFL_G=(\FL_{G}\times\Spec(\Qp))^{\an}$, as well as the groups $\cG^{\an}=(G\times \Spec(\Qp))^{\an}$, $\cP^{\an}_{G}=(P_{G}\times \Spec(\Qp))^{\an}$ and $\cM_G^{\an}=(M_G\times \Spec(\Qp))^{\an}$.

  Since we have fixed an integral model of $G$, we have quasi-compact, affinoid groups $\cG\hookrightarrow \cG^{\an}$, $\cP_{G,\Sieg}\hookrightarrow \cP_{G,\Sieg}^{\an}$ and $\cM_G\hookrightarrow \cM_G^{\an}$.

  Write $\cS_{G,K^p}^{\tor}$ for the perfectoid space $\varprojlim_{K_p} \cS_{G,K^pK_p}^{\tor}$. We can then consider the Hodge--Tate period map
  \[ \pi_{\HT,G}^{\tor}: \cS_{G,K^p}^{\tor} \longrightarrow \tFL_G\]
  which for every open compact $K_p\subset G(\Qp)$ descends to a map of topological spaces (c.f. \cite[\S 4.5]{boxerpilloni20})
  \[ \pi_{\HT,G,K_p}^{\tor}: \cS_{G,K^pK_p}^{\tor}\longrightarrow \tFL_G/K_p.\]

 There is an analogous Hodge--Tate period map for $H$ also. To lighten the notation, we shall frequently omit many of the subscripts from $\pi_{\HT,G,K_p}^{\tor}$ when they are clear from context (in particular, we shall almost always omit the ``$\tor$'', since the non-compactified Shimura variety plays no role here).

  We shall attempt to consistently maintain the convention that subsets of flag varieties are denoted by typewriter letters $\tU$ etc, and the preimages of these spaces under the Hodge--Tate period maps are denoted by calligraphic letters $\cU$ etc.


 \subsection{Period maps}

  \begin{theorem}
   \label{thm:HTmaps}
   There is a commutative diagram of Hodge--Tate period maps (where we have omitted some unimportant subscripts and superscripts for clarity)
   \[
    \begin{tikzcd}
    \cS^{\tor}_{H,\Iw}(p^t) \rar["\pi^H_{\Iw}"]  & \tFL_H / K_{H, \Iw}(p^t)\\
    \cS^{\tor}_{H,\ds}(p^t) \rar["\pi^H_{\ds}"] \dar["\hat\iota" left] \uar["\pr_\ds" left] & \tFL_H / K_{H, \ds}(p^t) \dar["\hat\iota" right] \uar["\pr_{\ds}" right]\\
    \cS^{\tor}_{G,\Iw}(p^t) \rar["\pi^G_{\Iw}"] & \tFL_G / K_{G, \Iw}(p^t)
    \end{tikzcd}
   \]
   in which the maps $\pr_{\ds}$ are the natural quotients, and the downward ones are given by the composite of the natural embedding $H \into G$ and right-translation by $\hat\gamma = \gamma w_1$.
  \end{theorem}

  \begin{proof}
   It suffices to check that the Hodge--Tate period maps for $H$ and for $G$ at perfectoid infinite level are compatible; but this is a direct consequence of the construction, since the Hodge--Tate period map for Hodge-type Shimura varieties is defined using an embedding into a Siegel Shimura variety. See \cite[\S 4.4.7]{boxerpilloni20}).
  \end{proof}

  For $(m, n, t)$ as in \eqref{eq:mnt}, we define subspaces
  \[ \cI^G_{m, n} \subset \cU^G_n \subset \cS_{G,\Iw}(p^t), \qquad \cZ^H_{m} \subset \cU^H_n \subset \cS_{H,\ds}(p^t)
  \]
  as the preimages of the subsets $\tI^G_{m, n} \subset \tU^G_n \subset \tFL^G$ under $\pi^G_{\Iw}$, respectively $\cZ^H_{m} \subset \cU^H_n \subset \tFL^H$ under $\pi^H_{\ds}$. Combining \cref{thm:HTmaps} and \cref{eq:Cartadic2}, we obtain a Cartesian diagram
  \begin{equation}\label{eq:adicSh}
   \begin{tikzcd}
    \cZ^{H}_m \dar \rar[hook] &\cU^H_n \dar\\
    \cI^G_{m, n} \rar[hook] & \cU^G_n
   \end{tikzcd}
  \end{equation}
  in which the horizontal arrows are closed embeddings; and $\cZ^{H}_m$ is closed in $\cS_{H,\ds}(p^t)$.


 \subsection{Overconvergent pullback}
  We can now define the pullback map on overconvergent cohomology. We give the definitions for the non-cuspidal cohomology, using the coefficient sheaf $\cV= \cV_\kappa$ for some $M_G$-dominant integral weight $\kappa$; the definitions are the same for cuspidal cohomology using $\cV= \cV_\kappa(-D)$ instead.

  Using the diagram \eqref{eq:adicSh} and the functoriality of cohomology with support, we get a map
  \begin{equation}
   \label{eq:iota1}
   \hat\iota^* : R\Gamma_{\cI^G_{mn}}(\cU_n^G, \cV) \longrightarrow R\Gamma_{\cZ_m^{H}}\left(\cU_n^H, \hat\iota^* \cV_{\kappa}\right) \cong R\Gamma_{\cZ_m^{H}}\left(\cS_{H,\ds}(p^t), \hat\iota^* \cV\right),
  \end{equation}
  where the final isomorphism comes from excision, using the fact that $\cZ_m^{H}$ is closed in $\cS_{H,\ds}(p^t)$.

  As in \bp{\S 5.4.1}, for any $t \ge 1$, we can define the finite-slope overconvergent cohomology for $G$ as
  \begin{align*}
    R\Gamma^G_{w_1}(\kappa)^{-, \fs} &\coloneqq R\Gamma_{\cI_{00}}(\cU^G_0, \cV_\kappa)^{-, \fs}.
  \end{align*}
  So \eqref{eq:iota1} for $(m, n) = (0,0)$ gives our first definition of the pullback map on overconvergent cohomology, taking values in $R\Gamma_{\cZ_0^{H}}\left(\cS_{H,\ds}(p^t), \hat\iota^* \cV\right)$.

  \begin{proposition}[Comparison with classical pullback]
   \label{prop:classicalcomp}
   We have the following commutative diagram:
   \[
    \begin{tikzcd}
     R\Gamma_{\cI_{00}^G}(\cU_0^G,\cV) \rar["\quad \hat\iota^*\quad" above, "\eqref{eq:iota1}" below] \dar["\cores"] & R\Gamma_{\cZ_0^{H}}\left(\cS^{\tor}_{H, \ds}(p^t), \hat\iota^* \cV\right) \dar["\cores"]\\
     R\Gamma(\cU_0^G,\cV) \rar{\quad \hat\iota^*\quad }  & R\Gamma\left(\cS^{\tor}_{H,\ds}(p^t),\hat\iota^*\cV\right)\\
        R\Gamma\left(\cS^{\tor}_{G,\Iw}(p^t),\cV\right)\uar["\res"] \arrow["\hat\iota^*" below]{ru}
    \end{tikzcd}
   \]
   in which the bottom horizontal map corresponds to the classical pushforward via the rigid-analytic GAGA theorem. Moreover, the spaces in the left column of the diagram have actions of the prime-to-$p$ Hecke algebra and the operators $\cU'_{\Sieg}$, $\cU'_{\Kl}$ at $p$, and the maps $\res$ and $\cores$ are compatible with these actions.
  \end{proposition}

  \begin{proof}
   The only non-obvious step of the diagram is the existence of the middle horizontal map, which follows from \cref{prop:Cartadic}. The compatibility with Hecke actions is an easy check, cf.~\bp{Lemma 5.17}.
  \end{proof}

  \begin{proposition}[Change of support condition]
   \label{prop:changesupport}
   The maps \eqref{eq:iota1} for $(m, n)$, $(m, 0)$, and $(0, 0)$ fit into a diagram
   \[
    \begin{tikzcd}
     R\Gamma_{\cI_{mn}^G}(\cU_n^G,\cV) \arrow{r}{\hat\iota^*} &
     R\Gamma_{\cZ_m^H}\left(\cS^{\tor}_{H, \ds}(p^t), \hat\iota^*(\cV)\right)\\
     R\Gamma_{\cI_{m0}^G}(\cU_0^G,\cV) \arrow{u}{\res} \arrow{r}{\hat\iota^*} \dar["\cores"]& R\Gamma_{\cZ_m^H}\left(\cS^{\tor}_{H, \ds}(p^t), \hat\iota^*(\cV)\right) \uar[equals] \arrow{d}{\cores} \\
     R\Gamma_{\cI_{00}^G}(\cU_0^G,\cV) \arrow{r}{\hat\iota^*} &
     R\Gamma_{\cZ_0^H}\left(\cS^{\tor}_{H, \ds}(p^t), \hat\iota^*(\cV)\right).
    \end{tikzcd}
   \]
   The complexes in the left column have compatible actions of the Hecke operators away from $p$, and of $\cU'_{\Kl}$ and $\cU'_{\Sieg}$ at $p$, and the maps $\res$ and $\cores$ are compatible with these.
  \end{proposition}

  \begin{proof}
   Immediate from the fact that $\cI^G_{mn} = \cI_{m0}^G \cap \cU_n^G$ and standard functoriality properties of cohomology with support.
  \end{proof}

  \begin{proposition}
   The spaces in the left columns of the diagrams in \cref{prop:classicalcomp,prop:changesupport} all have actions of the prime-to-$p$ Hecke operators, and of the Hecke operators $\cU'_{\Sieg}$, $\cU'_{\Kl}$, $\cU'_B$ at $p$. Moreover, the maps in the left column of \cref{prop:changesupport} become isomorphisms on the finite-slope part for $\cU'_B$.
  \end{proposition}

  \begin{proof}
   The compatibility with Hecke operators away from $p$ is clear, since the Hodge--Tate period map is invariant under the action of the prime-to-$p$ Hecke algebra. The fact that the maps in the left column of \cref{prop:classicalcomp} are maps of Hecke modules is an instance of \bp{Lemma 5.17}.

   The assertions regarding the finite-slope part follow from \bp{Theorem 5.66}, since one can check that $(\cI_{m, n}, \cU_n)$ defines an ``allowed support condition'' in the sense of \bp{\S 5.4.3}.
  \end{proof}

  This shows that we have well-defined maps
  \begin{equation}
   \label{eq:iota2}
   R\Gamma^G_{w_1}(\kappa)^{-, \fs} \longrightarrow R\Gamma_{\cZ_m^H}\left(\cS^{\tor}_{H, \ds}(p^t), \hat\iota^*(\cV)\right)
  \end{equation}
  for any $0 \le m < t$, compatible under corestriction, extending \eqref{eq:iota1} for in the $m = 0$ case.


 \subsection{Functoriality of coefficients}
  \label{sect:branchcoeffs}
  \begin{proposition}
   Let $\kappa_1 = (r_1, -r_2-2; r_1 + r_2)$ with $r_1 \ge r_2 \ge -1$, and let $\tau = (t_1, t_2; r_1 + r_2)$ where $t_i \ge -1$ and $t_1 + t_2 = r_1 - r_2 - 2$. Then there is a nonzero homomorphism of $(\gamma^{-1} M_H \gamma)$-representations
   \[ V^G_{\kappa_1} |_{\gamma^{-1} M_H\gamma } \to V^H_{\tau}, \]
   uniquely determined up to scaling.
  \end{proposition}

  It will be helpful to fix a normalisation for this map, by choosing a vector $f \in (V^G_{\kappa_1})^\vee$ which transforms by $\tau^{-1}$ under $\gamma^{-1} M_H \gamma$. We have an explicit presentation of $(V^G_{\kappa_1})^\vee = V^G_{(-w_{0, M} \kappa_1)}$ as the space of polynomial functions $f \in \cO(M_G)$ which satisfy $f(mb) = \kappa_1(b) f(m)$ for all $b \in B_{M_G}$ and $m \in M_G$, with $M_G$ acting by left-translation. Since $\gamma^{-1} M_H \gamma \cdot B_{M_G}$ is open in $M_G$, we can choose a unique $f$ which satisfies $f(\mathrm{id}) = 1$ and transforms via $\tau^{-1}$ under the action of $M_H$.

  This map gives a homomorphism of sheaves on $\cS^{\tor}_{H, \ds}(p^t)$,
  \[ \hat\iota^*\left(\cV_\kappa^G\right) \longrightarrow \cV_\tau^H, \]
  and combining this with \cref{eq:iota2}, we obtain maps of complexes
  \[ R\Gamma^G_{w_1}(\kappa_1)^{-, \fs} \longrightarrow R\Gamma_{\cZ_m^H}\left(\cS^{\tor}_{H, \ds}(p^t), \cV^H_\tau\right)
  \]
  for any $\tau$ in the appropriate range, and similarly for cuspidal cohomology.

  \begin{remark}
   The map is formally well-defined for a rather wider range of values of the parameters; but we have restricted to the case when $r_1 \ge r_2 \ge -1$ and $t_1, t_1 \ge -1$, in order that there are interesting cuspidal automorphic representations contributing to $H^2$ for both $\cV_{\kappa_1}^G$ and $\cV^H_{\tau}$.

   If we set $k_i = r_i + 3$ and $c_i = t_i + 2$, then $(k_1, k_2, c_1, c_2)$ will define a point lying on the top edge of the region labelled $(f)$ in \cite[Diagram 2]{LZvista}. Unfortunately, it seems to be difficult to extend our present analysis to points in the interior of this region; this would require some sort of ``nearly version'' of higher Coleman theory, analogous to the theory of nearly-overconvergent families in $H^0$ of modular curves recently introduced by Andreatta--Iovita \cite{andreattaiovita21}. See \cite[\S 6]{LPSZ1} for an analogous theory in the ordinary case (with $r_2$ fixed, rather than varying as here).
  \end{remark}

 \subsection{Change of level}
  \label{sect:changelevel}
  Finally, we note that for any $t \ge 1$, we have $[K_{H, \ds}(p^t): K_{H, \ds}(p^{t+1})] = p^4 = [K_{G, \Iw}(p^t) : K_{G, \Iw}(p^{t+1})]$, and hence the natural map
  \[ \cS_{H, \ds}(p^{t+1}) \longrightarrow \cS_{H, \ds}(p^{t}) \times_{\cS_{G, \Iw}(p^{t})} \cS_{G, \Iw}(p^{t+1})
  \]
  is an isomorphism. So the pushforward (trace) maps arising from changing $t$ on the two spaces are compatible with the pullback $\hat\iota^*$, and similarly for the cohomology with supports, for any support condition invariant under $K_{G, \Iw}(p^t)$.

  Hence, if we temporarily write $\cI_{mn}^G(p^t)$ etc to distinguish our various locally closed subspaces of Shimura varieties at the different levels, then we have trace maps
  \[ R\Gamma_{\cI^G_{mn}(p^{t+1})}(\cU^G_n(p^{t+1}), \cV^G_\kappa) \longrightarrow R\Gamma_{\cI^G_{mn}(p^t)}(\cU^G_n(p^t), \cV^G_\kappa) \]
  and
  \[  R\Gamma_{\cZ^H_{m}(p^{t+1})}(\cS_{H, \ds}(p^{t+1}), \cV^H_\tau) \longrightarrow R\Gamma_{\cZ^H_{m}(p^t)}(\cS_{H, \ds}(p^{t}), \cV^H_\tau).\]
  and these are compatible with the pullback maps $\hat\iota^*$, and the restriction/corestriction maps for varying $m, n$. Moreover, the trace maps for $G$ are isomorphisms on the finite-slope part by \bp{Theorem 5.14}.

  We can thus define a map
  \[
   R\Gamma_{w_1}(\kappa)^{-, \fs} \longrightarrow R\Gamma_{\cZ^H_m(p^t)}(\cS_{H, \ds}(p^t), \cV^H_\tau).
  \]
  for any $t \ge 1$, $m \ge 0$ (not necessarily with $t > m$) by composing with the trace map from level $t'$ for some auxiliary $t' > m$; this allows us to define $\hat\iota^*$ as a map
  \[
   R\Gamma^G_{w_1}(\kappa)^{-, \fs} \to \varprojlim_m R\Gamma_{\cZ^H_m(p^t)}(\cS_{H, \ds}(p^t),  \cV^H_\tau).
  \]
  Since the spaces $\cZ_m^H$ are actually invariant under the Iwahori of $H$, we can trace down further to land in the space
  \[ \varprojlim_m R\Gamma_{\cZ^H_m(p^t)}(\cS_{H, \Iw}(p^t),  \cV^H_\tau),\]
  where we have abused notation a little by using $\cZ^H_m(p^t)$ for the preimages of $\tZ_m^H$ at either Iwahori or $\ds$ level.
  \begin{definition}
   We define
   \[ R\Gamma_{\mathrm{id}}\left(\cS_{H, \Iw}(p^t), \tau\right)^{(-, \dag)} = \varprojlim_m R\Gamma_{\cZ^H_{m}(p^t)}\left( \cS_{H, \Iw}(p^t), \cV^H_\tau\right) .\]
  \end{definition}

  This space can be interpreted as the compactly-supported cohomology of the intersection $\bigcap_m \cZ^H_m(p^t) = \pi_{H}^{-1}\left( \{\mathrm{id}_H\} \right)$; we shall recall this in a little more detail in the next section, where we shall allow more general coefficients.


\section{Torsors}

 We now begin constructing the ``locally analytic'' version of the pullback map on higher Coleman theory.
 \subsection{Torsors on flag varieties}

  The map $x \mapsto x^{-1}: G\rightarrow \FL_G$ (recall that $\FL_G=P_G\backslash G$) allows us to regard $G$ as a right $P_G$-torsor over $\FL_G$, and similarly to regard $G/ N_G\rightarrow \FL_G$ as a right $M_G$-torsor. We consider their analytifications
  \[ \tP^G: \cG\rightarrow \tFL_G \qquad \text{and}  \qquad \tM^G:\cG/ \cN_G\rightarrow \tFL_G.\]
  which are torsors over $\tFL_G$ under the (affinoid) analytic groups $\cP_G$ and $\cM_G$ respectively. We similarly define torsors over the flag varieties of $H$ and $H_i$ for $i=1,2$.

  \begin{definition} Define $\cP^G_{\HT}$ and $\cM^G_{\HT}$ to be the pullbacks via $\pi^G_{\HT}$ of the torsors $\tP^G$ and $\tM^G$; these are (right) torsors over $\cS_{G,\Iw}(p^t)$ for the groups $\cP_G$ and $\cM_G$. We similarly define $\cP^H_{\HT}$ and $\cM^H_{\HT}$, $\cP^{H_i}_{\HT}$ and $\cM^{H_i}_{\HT}$ for $i=1,2$.
  \end{definition}

  \begin{note}
   It is easy to check that $\cM^H_{\HT}=\cM^{H_1}_{\HT}\times_{\GL_1}\cM^{H_2}_{\HT}$, where we take the fibre product with respect to the action of $\nu$ in the parametrisation of $T$.
  \end{note}


 \subsection{Reduction of structure}

  \begin{definition}
   For $n > 0$, let $\cM^1_{G,n} \triangleleft \cM_G$ be the (affinoid analytic) group of elements which reduce to the identity $\pmod{p^n}$. Define
   \[ \cM^\square_{G,n} = \cM^1_{G,n} \cdot B_{M_G}(\Zp), \]
   which is an affinoid analytic subgroup containing $\Iw_{M_G}(p^n)$. A similar definition applies to $M_H = T$; we write the group as $\cT^{\square}_{n} = T(\Zp) \cT^1_{n}$.
  \end{definition}

  \begin{note}
   We follow \cite{boxerpilloni20} here in using affinoid subgroups and affinoid subspaces of flag varieties to develop the locally-analytic theory, rather than the ``mixed'' spaces (products of some copies of $\cB_n$ and some of $\cB^\circ_n$) used in the previous sections.
  \end{note}

  \begin{note}\label{note:identifyMGsq}
   Identifying $M_G$ with $\GL_2 \times \GL_1$ as in the introduction, we have
   \[\cM^\square_{G,  n} =
    \left\{
     (\stbt x y z w, \lambda) :
     \begin{array}{c}
       x,w,\lambda \in \Zp^\times \cdot (1 + \cB_n),\\ z \in \cB_n, y \in \Zp + \cB_n.
     \end{array}
    \right\}
    \qedhere
   \]
  \end{note}

  \begin{notation}
   Define
   \[ \cT^\ds_n = \{\diag(t_1, t_2, \nu t_2^{-1}, \nu t_1^{-1}) \in \cT^\square_n: t_1-t_2 \in \cB_n\}. \]
  \end{notation}

  Thus $\cT^\ds_n$ and $\cT^\square_n$ are both disjoint unions of copies of $\cT^1_n$, but $\cT^\ds_n$ has fewer of these components than $\cT^\square_n$.

  \begin{proposition} \label{prop:redofstr} Let $t > n > 0$.
   \begin{enumerate}
    \item Over $\cU^G_n$, the torsor $\cM_{\HT}^G$ has a reduction of structure to an \'etale torsor $\cM^G_{\HT,n}$ under the group $\cM^\square_{G,n}$.

    \item Over $\cU^H_{\Iw,n}$, the torsor $\cM_{\HT}^H$ has a reduction of structure to an \'etale torsor $\cM^H_{\HT,n,\Iw}$ under the group $\cT^{\square}_{n}$.

    \item Over $\cU^H_n$, the torsor $\cM_{\HT}^H$ has a reduction of structure to an \'etale torsor $\cM^H_{\HT,n,\ds}$ under the group $\cT^\ds_{n}$ (and this refines the pullback of $\cM^H_{\HT,n,\Iw}$ to level $K^H_{\ds}(p^t)$).

  \end{enumerate}
  \end{proposition}

  \begin{proof}
   Part (1) is essentially the result of \bp{\S 6.2.1}. The proofs of (2) and (3) are similar.
  \end{proof}

  \begin{lemma}
   We have the following inclusions of subgroups.
   \begin{itemize}
    \item As subgroups of $\cM_G$, we have
    \[ \cT^\ds_{n} = \cT \cap \gamma \cM^\square_{G,n}\gamma^{-1}. \]
    \item As subgroups of $\cG$,
    \[ \hat\gamma^{-1} \cdot K^H_{\ds}(p^n)\cH^1_{n} \cdot \hat\gamma \subset K^G_{\Iw}(p^n)\cG^1_{n}, \]
    where $\hat\gamma = \gamma w_1$ as usual.
   \end{itemize}
  \end{lemma}

  \begin{proof}
   If $\tau = (\stbt{t_1}{}{}{t_2}, \nu)$ is an element of $\cT$ then $\gamma^{-1} \tau \gamma = ( \stbt{t_1}{}{t_2-t_1}{t_2}, \nu)$. It is now clear that $\gamma^{-1}T_{\ds}(p^t) \gamma \subset \Iw_{M_G}(p^t)$ and $\gamma^{-1}\cT^1_{n} \gamma \subset \cM^1_{G,n}$, so the required inclusion follows.

   The second statement can be verified similarly; the inclusion on $\Zp$-points is the definition of $K^H_{\ds}$, and the inclusion on $\cH^1_{n}$ follows from the fact that $\cH^1_{n} \subset \cG^1_n$ and $\cG^1_n$ is normal in $\cG$.
  \end{proof}


  \begin{proposition}\label{prop:pullbackcomp}
   We have an equality of $\cM^\square_{G, n}$-torsors over $\cU^H_{n,\ds}$:
   \[ \hat\iota^*\left( \cM^G_{\HT,n,\Iw}\right) = \cM^H_{\HT,n,\ds}\times^{\left[\cT^\ds_{n}, \gamma\right]} \cM^\square_{G,n}, \]
   where we regard $\cT^{\ds}_{n}$ as a subgroup of $\Iw_{M_G}(p^t)\cM^1_{G,n}$ via conjugation by $\gamma$.
  \end{proposition}

  \begin{proof}
   We check the analogous statement on the flag varities. We first observe that we have a commutative diagram of adic spaces
   \[
    \begin{tikzcd}
     K^H_{\ds}(p^t)\cH^1_{n} \dar\rar & K^G_{\Iw}(p^t)\cG^1_{n}\dar\\
     \cB^H\backslash \cB^H K^H_{\ds}(p^t)\cH^1_{n}\rar[hook, "\hat\iota"]
     & \cP^G\backslash \cP^G w_1K^G_{\Iw}(p^t)\cG^1_{n}.
    \end{tikzcd}
   \]
   Here, the vertical maps are given by $h\mapsto \cB^H\backslash \cB^H h^{-1}$ on the left, and $g\mapsto \cP^G\backslash \cP^G w_1g^{-1}$ on the right; the lower horizontal map $\hat\iota$ is $\cB^H h\mapsto \cP^G h \gamma w_1$, and the map along the top making the diagram commute is $h \mapsto \hat\gamma^{-1} h\hat\gamma$, which is well-defined by the preceding lemma. (Note that the commutativity of the diagram relies on the fact that $\gamma \in P_G$.)

   The right-translation action of $\cB^H$ on $\cG$ makes the left-hand column into a torsor for the group $\cB^H \cap K^H_{\ds}(p^t)\cH^1_{n}$. Similarly, via right-translation conjugated by $w_1$, $K^G_{\Iw}(p^t)\cG^1_{n}$ becomes a torsor for the group
   \(  \cP\cap  w \cG^1_n K^G_{\Iw}(p^t)w^{-1}; \)
   and these structures are compatible if we consider $\cB^H \cap K^H_{\ds}(p^t)\cH^1_{n}$ as a subgroup of $\cP \cap w \cG^1_n K^G_{\Iw}(p^t)w^{-1} $ via conjugation by $\gamma$.

   Passing to the $\cN_H$-coinvariants on the left, we obtain a torsor for $T_\ds(p^t) \cT^1_{n} = \cT^\ds_n$; and passing to $\cN_G$-coinvariants on the right, we obtain a torsor for the projection of $\cP \cap w \cG^1_n K^G_{\Iw}(p^t)w^{-1}$ to the Levi $\cM_G$, which is the group $\cM^\square_{G, n}$. Moreover, these structures are compatible via the $\gamma$-conjugation inclusion $\cT^\ds_n \into \cM^\square_{G,n}$ established in the above lemma.

   We now note that $\tU^G_n \subset \tFL^G$ is contained in the subset $\cP^G w_1K^G_{\Iw}(p^t)\cG^1_{n}$, since $\tU^G_n$ is the orbit of $w_1$ under $K^G_{\Iw}(p^t) \cG^1_{n, n}$ in the notation of \bp{\S 3.3.3}, and $\cG^1_{n, n} \subset \cG^1_n$. So pulling back to $\cU^G_n$ via the Hodge--Tate period map gives the result.
  \end{proof}

\section{Spaces of distributions and branching laws}

 \subsection{Analytic characters}

  \begin{definition}
   Let $n \in \QQ_{> 0}$. We say a continuous character $\kappa: \Zp^\times \to A^\times$, for $(A, A^+)$ a complete Tate algebra, is \textbf{$n$-analytic} if it extends to an analytic $A$-valued function on the affinoid adic space
   \[ \Zp^\times \cdot \cB_n \subset \mathbf{G}_m^{\mathrm{ad}}.\]
   This definition extends naturally to characters $T(\Zp) \to A^\times$: the $n$-analytic characters are exactly those which extend to $\cT^\square_n$.
  \end{definition}

  \begin{remark}
   For compatibility with our notations for algebraic weights, we shall denote a $p$-adic character $\kappa$ of $T(\Zp)$ by a triple $(\rho_1, \rho_2; \omega)$ of characters of $\Zp^\times$, via
   \[ \kappa(\diag(st_1, st_2, st_2^{-1}, st_1^{-1})) = \rho_1(t_1) \rho_2(t_2) \omega(s), \]
   so that formally
   \[ \kappa(\diag(t_1, t_2, \nu t_2^{-1}, \nu t_1^{-1})) = \rho_1(t_1) \rho_2(t_2) \left(\frac{\omega}{\rho_1\rho_2}\right)^{\tfrac{1}{2}}(\nu).\]
   This is of course not well-defined as written, since $p$-adic characters do not have a unique square root, so we should understand the triple $(\rho_1, \rho_2; \omega)$ as coming with an implicit choice of square root of $\omega /\rho_1 \rho_2$ which is being suppressed from the notation.
  \end{remark}

 \subsection{Analytic inductions}

  We recall some definitions from \bp{\S 6.1.2}. Let $(A,A^+)$ be a complete Tate algebra over $(\Qp,\Zp)$. Let $n_0 > 0$, and assume that $\kappa_A: T(\Zp)\rightarrow A^\times$ is an $n_0$-analytic character.  For $?\in \{G,H\}$ and $n \geq n_0$, let $\cM^1_{?,n}$ be the affinoid subgroup of $\cM_?$ defined above, and let $B_{M_G}$ be the Borel of $M_?$.

  \begin{definition}
   For $n\geq n_0$, define
   \begin{align*}
    V^{n-\an}_{G,\kappa_A}=&\, \an\Ind^{\left(\cM^\square_n\right) }_{\left(\cM^\square_n \cap \cB_{G}\right)}(w_{0,M_?}\kappa_A)\\
    =&\, \Big\{ f \in \cO(\cM^\square_{G, n}) \htimes A :  f(mb)=(w_{0,M}\kappa_A)(b^{-1})f(m),\\
     &\quad \forall \, m\in \cM^\square_{G, n},\, \forall b\in \cM^\square_{G, n} \cap \cB_G \Big\}.
   \end{align*}
   We define a left action of $\cM^\square_{G, n}$ on $V^{n-\an}_{G,\kappa_A}$ by $(h \cdot f)(m) = f(h^{-1} m)$.

   Write $D^{n-\an}_{G,\kappa_A}$ for the dual space, and $\langle -, - \rangle$ for the pairing between these; we equip $D^{n-\an}_{G,\kappa_A}$ with a left action of the same group $\cM^\square_{G, n}$, in such a way that $\langle h\mu , hf \rangle = \langle \mu ,f \rangle$.
  \end{definition}

  Let us describe $V_{G,\kappa_A}^{n-\an}$ explicitly. We use the description of $\cM_{G, n}^{\square}$ given in \cref{note:identifyMGsq}. Since $\cM_{G, n}^{\square}$ has an Iwahori decomposition, restriction to elements of the form $(\stbt{1}{}{\star}{1}, 1)$ identifies $V_{G,\kappa_A}^{n-\mathrm{an}}$ with the space of analytic functions of $z \in \cB_n$; this space is independent of $\kappa_A$, but the action of $\cM^{\square}_{G,n}$ does depend on $\kappa_A$, as follows.

  \begin{propqed}
   Suppose $\kappa_A$ is the character $(\rho_1, \rho_2; \omega)$. Then the action of $(\stbt a b c d, \nu)$ on $f \in \cO(\cB_n) \htimes A$ is given by
   \[ \left((\stbt a b c d, \nu) f\right)(z) = f\left( \frac{az-c}{-bz + d}\right) (-bz + d)^{\rho_1 - \rho_2} (ad-bc)^{\rho_2} \nu^{(\omega-\rho_1-\rho_2)/2}.\qedhere\]
  \end{propqed}

  \begin{note}
   For $H$ in place of $G$, we can make the same definitions; but the resulting spaces are much simpler, since $\cM_H = \cT$ is commutative and contained in $\cB_H$. Hence any function $f \in V^{n-\an}_{H,\kappa_A}$ is uniquely determined by its value at 1. So $V^{n-\an}_{H,\kappa_A}$ is canonically $A$, with $\cT^\square_n$ acting via $\kappa_A$; and dually $D^{n-\an}_{H,\kappa_A}$ is $A$ with $\cT^\square_n$ acting via $\kappa_A^{-1}$.
  \end{note}

  \begin{note}
   If $\kappa_A$ is an algebraic character $(k_1, k_2; c)$, then $V_{G,\kappa_A}^{n-\an}$ naturally contains the algebraic $M_H$-representation of highest weight $\kappa_A$ (identified with polynomials in $z$ of degree $\le k_1 - k_2$); and dually, $D^{n-\an}_{?,\kappa_A}$ surjects onto the algebraic representation of highest weight $\kappa_A^\vee$ (the dual of the weight $\kappa_A$ representation).
  \end{note}
%
%
%


 \subsection{Branching laws in families}\label{ss:kraken}

  \begin{definition}
   Let $A$ be a Tate algebra endowed with an $n_0$-analytic character $\kappa_A: T(\Zp) \to A^\times$ as above, and additionally with a character $\lambda: (1 + \cB_n)^\times \to A^\times$. Define the \emph{kraken} to be the function
   \[ \mathscr{K}^{\lambda}(z)=\lambda(1+z),\]
   viewed as an element of $V_{G,\kappa_{A}}^{n-\an}$.
  \end{definition}

  \begin{lemma}
   The function $\mathscr{K}^{\lambda}$ is an eigenvector for $\gamma^{-1} \cT^\ds_n \gamma \subset \cM^\square_{G, n}$, with eigencharacter $w_{0, M} \kappa_A + (\lambda, -\lambda; 0)$.
  \end{lemma}

  \begin{proof}
   We have $\gamma^{-1}\left( \stbt x {} {} y, \nu\right) \gamma=(\stbt x {} {-x+y} y, \nu)$.
   If this condition is satisfied, then (writing $\kappa = (\rho_1, \rho_2; \omega)$ as before) we have
   \[ (\stbt x {} {-x+y} y, \nu) \mathscr{K}^{\lambda}(z)=  x^{\rho_2}y^{\rho_1}\nu^{(\omega-\rho_1-\rho_2)/2} \mathscr{K}^{\lambda}\left( \frac{x}{y}(z+1) - 1 \right)=x^{\rho_2+\lambda}y^{\rho_1-\lambda} \nu^{(\omega-\rho_1-\rho_2)/2}\mathscr{K}^{\lambda}(z).\qedhere \]
  \end{proof}

  As an immediate consquence, we obtain the following result:

  \begin{propqed}\label{prop:krakenpower}
   Pairing with the element $\mathscr{K}^{\lambda}$ defines a homomorphism of $\cT^\ds_n$-representations
   \[ \hat\iota^*( D^{n-\an}_{G,\kappa_A}) \longrightarrow D^{n-\an}_{H,w_{0, M} \kappa_A + (\lambda, -\lambda;0)}.\qedhere\]
  \end{propqed}

  \begin{note}
   Note that $D^{n-\an}_{H,w_{0, M} \kappa_A + (\lambda, -\lambda;0)}$ is one-dimensional (and independent of $n$).
  \end{note}

  We now consider a special case. Let $A = \Qp$ and take $\kappa_A$ to be the algebraic weight $(r_2 + 2, -r_1; -r_1-r_2)$, for some integers $r_1 \ge r_2\ge 0$, so that $\kappa_A^\vee$ is the weight $\kappa_1$ of \eqref{eq:ourweights}. If we choose $\lambda$ to be an integer in the range $[0, r_1 + r_2 + 2]$, then $\mathscr{K}^\lambda$ lies in the polynomial subspace $V_{G, \kappa_A} \subset V_{G, \kappa_A}^{n-\an}$. Its value at the identity element of $\cM_{G, n}^{\square}$ is 1, by definition.

  So, if $t_i \ge -1$ are integers with $t_1 + t_2 = r_1 - r_2 - 2$, and we we take $\lambda$ such that $(r_1 - \lambda, \lambda - 2 - r_2) = (t_1, t_2)$, then we obtain a commutative diagram of $\cT^\ds_n$-representations
  \[\begin{tikzcd}
   \iota^*(D^{n-\an}_{G,\kappa_A}) \rar \dar &D^{n-\an}_{H,-\tau_A} \dar["\cong"] \\
   \iota^*\left(V_{G,\kappa_A^\vee}\right) \rar &V_{H, \tau_A}
  \end{tikzcd}\]
  where $\tau_A^\vee = (t_1, t_2; r_1 + r_2)$. Hence the homomorphism of Proposition \ref{prop:krakenpower} is compatible with the classical branching law described in \cref{sect:branchcoeffs}


\section{Sheaves of distributions}

 We use the above function spaces and morphisms as ``models'' for sheaves on the Shimura variety.

 \subsection{Labelling of weights}
  \label{ss:analyticweights}

  We recall some definitions from \bp{\S 6.2} (this theory is a bit messy owing to the need to reconcile various different conventions).

  As above, we let $(A,A^+)$ be a Tate algebra over $(\Qp,\Zp)$. Given a weight $\nu_A : T(\Zp) \to A^\times$ for some coefficient ring $A$, following \bp{\S 6}, we define $\kappa_A: T(\Zp) \to A^\times$ by
  \[ \kappa_A = -w_{0, M} w_1 (\nu + \rho) - \rho. \]
  Explicitly, if $\nu_A$ is $(\nu_1, \nu_2; \omega)$ for some $\nu_i, \omega: \Zp^\times \to A^\times$, then
  \[ \kappa_A = (\nu_2 - 1, -3-\nu_1; -\omega).\]

  We are not so much interested in the linear dual $\kappa_A^\vee$ as the ``Serre dual'' $\kappa_A' = (\kappa_A + 2\rho_{nc})^\vee$. Explicitly this is $(\nu_1, -2-\nu_2; c) = w_1(\nu_A+ \rho) - \rho$. So when $A = \Qp$ and $\nu = (r_1, r_2; r_1 + r_2)$ is an integral algebraic weight, we have $\kappa_A' = \kappa_1$ in the notation of \eqref{eq:ourweights}.


 \subsection{Sheaves on $G$}

  Let $1 \le n < t$ be integers.

  \begin{definition}
   We now define two sheaves $\cV^{n-\an}_{G,\nu_A}$ and $\cD^{n-\an}_{G,\nu_A}$ over $\cU^G_n$. The former can be defined as a subsheaf of $\pi_* (\cM^G_{\HT, n,\Iw})$ transforming like functions in $V^{n-\an}_{\kappa_A}$; an alternative, possibly cleaner description is as a coproduct
   \[ \cV^{n-\an}_{G, \nu_A} = \cM^G_{\HT, n,\Iw} \times^{\cM^{\square}_{G, n}} V^{n-\an}_{G, \kappa_A},\]
   and similarly
   \begin{equation}\label{eq:Dsheafascofibreprod}
    \cD^{n-\an}_{G, \nu_A} = \cM^G_{\HT, n,\Iw} \times^{\cM^{\square}_{G, n}} D^{n-\an}_{G, (\kappa_A + 2\rho_{nc})}.
   \end{equation}
  \end{definition}

  (The shift by $2 \rho_{nc}$ is present so that the pairing between $\cD^{n-\an}_{G, \nu_A}$ and $\cV^{n-\an}_{G, \nu_A}$ lands in the dualizing sheaf of $\cS_G$, rather than in the structure sheaf.)

  \begin{lemma}\label{lem:sheafspec}
   The sheaves $\cV^{n-\an}_{G,\nu_A}$ and $\cD^{n-\an}_{G,\nu_A}$ are sheaves of $A$-modules, whose formation is compatible with base-change in $A$; and if $A = \Qp$ and $\nu_A = (r_1, r_2; c)$ for integers $r_1 \ge r_2 \ge -1$, we have classical comparison maps
   \[  \cV_{G, \kappa_A} \into \cV^{n-\an}_{G, \nu_A},
   \qquad
   \cD^{n-\an}_{G, \nu_A} \onto \cV_{G,(\kappa_A + 2\rho_{nc})^\vee} = \cV_{G, \kappa_1}. \]
  \end{lemma}

  \begin{proof}
   See \bp{Prop. 6.18}.
  \end{proof}

 \subsection{Sheaves on $H$}

  There are analogous constructions for sheaves for $H$. Here we use the element $\id \in {}^M W_H$ in place of $w_1$, and $w_{0, M_H}$ is the identity. So given an $n$-analytic character $\tau_A$, we define $\kappa_A^H = -\tau_A-2\rho_H$; and we set
  \[  \cV^{n-\an}_{H, \ds, \tau_A} = \cM^H_{\HT, n,\ds} \times^{\cT^{\ds}_n} V^{n-\an}_{H, \kappa_A^H},\]
  and
  \[ \cD^{n-\an}_{H, \ds, \tau_A} = \cM^G_{\HT, n, \ds} \times^{\cT^{\ds}_{n}} D^{n-\an}_{H,(\kappa_A^H + 2\rho_H)}.\]

  Thus $\cD^{n-\an}_{H,\ds,\tau_A}$ for a $\Qp$-valued algebraic character $\tau_A$ is simply (the restriction to $\cU_n^H$ of) the line bundle $\cV^H_{\tau_A}$. The same definitions make sense at Iwahori level, of course, giving line bundles $\cD^{n-\an}_{H,\Iw,\tau_A}$ and $\cV^{n-\an}_{H,\Iw,\tau_A}$. These sheaves are in fact independent of $n$ (in the sense that $\cD^{(n+1)-\an}_{H,\Iw,\tau_A}$ is isomorphic to the restriction of $\cD^{n-\an}_{H,\Iw,\tau_A}$ to $\cU_{n+1}^H$), so we shall frequently drop the $n$ and write simply $\cD^{\an}_{H,\Iw,\tau_A}$ etc.

  \begin{remark}\label{remark:bigsheafsmallt}
   Note that (for simplicity) we have only attempted to define the locally-analytic sheaves for $G$ when the level group at $p$ is $\Iw(p^t)$ with $t > n$; thus our functions are defined on $\cB_n$ itself, rather than on a union of translates of $\cB_n$. (This restriction on the levels is inherited from \bp{\S 6.3}.)

   However, for $H$ the technical difficulties disappear, and we can make sense of $\cV^{\an}_{H,\Iw,\tau_A}$ and $\cD^{\an}_{H,\Iw,\tau_A}$ as vector bundles on $\cU_{n, \Iw}^H(p^t)$ for any $n, t \ge 1$.
  \end{remark}

 \subsection{Branching for sheaves}

  \begin{definition}
   \label{def:compat}
   We say the $A$-valued, $n$-analytic characters $\nu_A$ and $\tau_A$ of $T(\Zp)$ are \emph{compatible} if $\nu_A = (\nu_1, \nu_2; \nu_1 + \nu_2)$, $\tau_A = (\tau_1, \tau_2; \nu_1 + \nu_2)$, for some characters $\nu_i, \tau_i$ of $\Zp^\times$, and we have the relation
   \[
    \tau_1 + \tau_2 = \nu_1 - \nu_2 - 2.
   \]
  \end{definition}

%
%

%
%
%

  Recall the kraken $\mathscr{K}^\lambda$ defined in \cref{ss:kraken}. If $\nu_A, \tau_A$ are compatible, then taking $\lambda = \nu_1 - \tau_1 = \nu_2 + \tau_2 + 2$, we obtain a homomorphism of $\cT^\ds_n$-representations
  \[ D^{n-\an}_{G,(\kappa_A + 2\rho_{nc})}
   \longrightarrow
   D^{n-\an}_{H,-\tau_A}\]
  where $\cT^{\ds}_n$ acts on $D^{n-\an}_{G,(\kappa_A + 2\rho_{nc})}$ via $\gamma$-conjugation. So the following result is an immediate consequence of \cref{prop:krakenpower} and the results of \cref{ss:analyticweights}:

  \begin{propqed}\label{prop:krakenonsheaves}
   Pairing with $\mathscr{K}^\lambda$ induces a morphism of sheaves over $\cU^H_{n}$:
   \[ \hat\iota^*( \cD^{n-\an}_{G,\nu_A})\longrightarrow \cD^{\an}_{H, \ds,\tau_A}.\]
   This morphism is compatible with specialisation in $A$, and if $A = \Qp$ and $\nu = (r_1, r_2; r_1+r_2)$, $\tau = (t_1, t_2; r_1 + r_2)$ are algebraic weights with $r_1 - r_2 \ge 0$ and $r_i, t_i \ge -1$, then this morphism is compatible with the map of finite-dimensional sheaves $\hat\iota^*\left(\cV_{\kappa_1}\right) \to \cV^H_{\tau}$ defined in \S\ref{sect:branchcoeffs}.
  \end{propqed}


 \subsection{Locally analytic overconvergent cohomology}

  Let $m,n,t$ be as in \cref{eq:mnt}, with $n > 0$; and suppose $\nu_A$ is an $n$-analytic $A$-valued character of $T(\Zp)$. We define cuspidal, locally analytic, overconvergent cohomology to be
  \begin{equation}\label{eq:cusplocanaoc}
   R\Gamma^G_{w, \an}(\nu_A, \cusp)^{-, \fs} = R\Gamma_{\cI_{mn}^G}\left(\cU^G_n, \cD^{n-\an}_{G,\nu_A}(-D_G)\right)^{-, \fs},
  \end{equation}
  and similarly for the non-cuspidal version. As shown in \bp{\S 6}, this complex is independent of $m$, $n$ and $t$, and is concentrated in degrees $[0, 1, 2]$.


  \begin{proposition}
   \label{prop:bigsheafpullback}
   Given $\nu_A$ and $\tau_A$ satisfying the compatibility condition of Definition \ref{def:compat}, we have a morphism of complexes of $A$-modules
   \[
    \hat\iota^*: R\Gamma^G_{w, \an}(\nu_A, \cusp)^{-, \fs}
    \to  R\Gamma_{\cZ^H_m}\left(\cU_n^H, \cD^{n-\an}_{H,\ds,\tau_A}(-D_H)\right).
   \]
  \end{proposition}

  \begin{proof}
   Immediate from \cref{prop:krakenonsheaves}.
  \end{proof}

  We have only defined this morphism for $m, n$ small relative to $t$. However, using \cref{remark:bigsheafsmallt}, we can argue as in \cref{sect:changelevel} and define
  \[ R\Gamma_{\id, \an}(\cS_{H, \Iw}(p^t), \tau_A, \cusp)^{-, \dag} = \varprojlim_m R\Gamma_{\cZ^H_{m, \Iw}(p^t)}\left( \cU^H_{n, \Iw}, \cD^{\an}_{H,\Iw,\tau_A}(-D_H)\right).\]
  Then we obtain a natural map
  \[R\Gamma^G_{w, \an}(\nu_A, \cusp)^{-, \fs} \to R\Gamma_{\id, \an}(\cS_{H, \Iw}(p^t), \tau_A, \cusp)^{-, \dag}.\]

  \begin{note}
   By construction, this morphism is compatible with derived base-change in $A$. If $A = \Qp$, and $\nu_A$ and $\tau_A$ are algebraic weights such that $r_1 \ge r_2 \ge -1$ and $t_1, t_2 \ge -1$, then this map fits into a commutative diagram with the pullback map on overconvergent cohomology defined in \S\ref{sect:changelevel}.
  \end{note}

 \subsection{Pairings and duality}

  Dually to the above, we define
  \[
   R\Gamma_{\id, \an}(\cS_{H, \Iw}(p^t), \tau_A)^{+, \dag} = \varinjlim_m R\Gamma\left(\cZ^H_{m, \Iw}(p^t), \cV^{\an}_{H,\Iw, \tau_A}\right).
  \]

  Note that if $\tau_A = (t_1, t_2; c)$ then $\cV^{\an}_{H,\Iw, \tau_A}$ is the sheaf $\cV^H_{-\tau-2\rho_H} = \cV_{(-2-t_1, -2-t_2; -c)}$, which is the sheaf of modular forms of weight $(t_1 + 2, t_2 + 2)$ (with the normalisation of the central character depending on $c$).

  \begin{proposition}
   The above complex is concentrated in degree 0 and independent of $t$. It can be identified with the space of $p$-adic overconvergent modular forms for $H$ of tame level $K^{H, p}$ and weight $\tau_A + (2, 2)$.
  \end{proposition}

  \begin{proof}
   For simplicity we suppose $K^{H, p}$ is the principal congruence subgroup of level $N$ for some $N$ (the general case reduces easily to this). Then the Shimura variety for $H$ is simply the fibre product (over $\mu_N$) of two copies of the level $N$ modular curve parametrising elliptic curves with full level $N$ structure and a cyclic subgroup of order $p^t$. Then $\pi_{HT}^{-1}(\{\id_H\})$ is the ``canonical locus'', where the $p$-subgroups are both multiplicative; and the $\cZ_{m, \Iw}^H$ are a cofinal family of neighbourhoods of this locus. Via the theory of the canonical subgroup, this space is independent of the choice of levels.

   Since the canonical locus is affinoid (and sufficiently small strict neighbourhoods of it also have this property), its cohomology vanishes above degree 0, and the degree 0 cohomology identifies with overconvergent sections of $\cV^{\an}_{H, \Iw, \tau_A}$. If we choose an extension $\tilde\tau_A$ of $\tau_A$ to the maximal torus of $\GL_2 \times \GL_2$, then $\cV^{\an}_{H, \Iw, \tau_A}$ decomposes as the product of two copies of the corresponding sheaves on the individual modular curves. This is precisely the construction of overconvergent modular forms described in \cite{pilloni13} (see the discussion following Prop 6.2 of \emph{op.cit.} for a comparison with Coleman's original approach).
  \end{proof}

  \begin{theorem}[c.f. \bp{Theorem 6.38}]
   \label{thm:bspairing}
   The cup product induces a pairing
   \[ H^2_{\id, \an}(\cS_{H, \Iw}(p^t), \tau_A, \cusp)^{-, \dag} \times H^0_{\id, \an}(\cS_{H, \Iw}(p^t), \tau_A)^{+, \dag} \longrightarrow A, \]
   whose formation is compatible with base-change in $A$, and which is compatible with the Serre duality pairing on classical cohomology when $A = \Qp$ and $\nu$, $\tau$ are classical weights.
  \end{theorem}

  \begin{proof}
   We define this pairing by combining the pullback map of \ref{prop:bigsheafpullback} with the pairing between the cohomology groups $H^2_{\id, \an}(\cS_{H, \Iw}(p^t), \tau_A, \cusp)^{-, \dag}$ and $H^0_{\id, \an}(\cS_{H, \Iw}(p^t), \tau_A, \cusp)^{+, \dag}$. By construction, this is compatible with Serre duality for each classical weight.
  \end{proof}



\section{Construction of the $p$-adic $L$-function}

 Let $L$ be a finite extension of $\Qp$.

 \subsection{Families of Eisenstein series}

  We refer to \cite[\S 7]{LPSZ1} for the construction of $p$-adic families of Eisenstein series $\cE^{\Phi^{(p)}}(\kappa_1, \kappa_2; \chi^{(p)})$, depending on a prime-to-$p$ Schwartz function $\Phi^{(p)}$ and prime-to-$p$ Dirichlet character $\chi^{(p)}$ (both valued in $L$) and a pair of characters $\kappa_1, \kappa_2$ of $\Zp^\times$ (valued in some $p$-adically complete $L$-algebra $A$).

  \begin{note}
   Note that this Eisenstein series is $p$-depleted, i.e.~lies in the kernel of $U_p$; and it is zero on any components of $\Spec(A)$ which do not satisfy the parity condition $\kappa_1(-1) \kappa_2(-1) = -\chi^{(p)}(-1)$.

   The construction factors through the projection of $\Phi^{(p)}$ to the eigenspace where $\stbt{a}{0}{0}{a}$ for $a \in \widehat{\ZZ}^{(p)}$ acts as $\widehat{\chi}^{(p)}(a)^{-1}$, where $\widehat{\chi}^{(p)}$ is the adelic character attached to $\chi^{(p)}$ as in \cite[\S 2.2]{LPSZ1}. We shall henceforth assume, without loss of generality, that $\Phi^{(p)}$ lies in this eigenspace; thus $\chi^{(p)}$ is uniquely determined by $\Phi^{(p)}$ and we sometimes drop it from the notation.
  \end{note}

  \begin{proposition}
   If $A$ is an affinoid algebra, and one of the $\kappa_i$ is a finite-order character, then $\cE^{\Phi^{p}}(\kappa_1, \kappa_2)$ is an overconvergent $A$-valued cusp form of weight-character $1 + \kappa_1 + \kappa_2$.
  \end{proposition}

  \begin{proof}
   Since twisting by a finite-order character preserves overconvergence, it suffices to assume $\kappa_1$ or $\kappa_2$ is 0. Then our $p$-adic Eisenstein series is the $p$-depletion of a family of \emph{ordinary} Eisenstein series, cf.~\cite[\S 2.3]{ohta99}, and it is well-known that these ordinary Eisenstein series are overconvergent (indeed, this is true by definition in Coleman's approach to overconvergent modular forms).
  \end{proof}

  As noted in \emph{op.cit.}, for $k \ge 1$, the Eisenstein series $F^{k}_{\Phi^p \Phi_{\mathrm{dep}}}$ described in \cite[\S 4.3]{LZ20b-regulator} is (the classical form associated to) $\cE^{\Phi^p}(k-1, 0)$, and $E^{k}_{\Phi^p \Phi_{\mathrm{dep}}}$ is $\cE^{\Phi^p}(0, k-1)$. It also implies the following relation:

  \begin{proposition}[cf.~{\cite[Prop 16.2.1]{LZ20b-regulator}}]
   Let $t \in \ZZ_{\ge 0}$. As overconvergent cusp forms of weight $-t$, we have
   \[ \theta^{-(1+t)}\left(F^{(t+2)}_{\Phi^p \Phi_{\mathrm{dep}}}\right) = \cE^{\Phi^p}(0, -1-t; \Phi^{(p)}), \]
   where $\theta = q \tfrac{\mathrm{d}}{\mathrm{d}q}$ is the Serre differential operator.
  \end{proposition}

 \subsection{Tame test data}

  We fix the following data:
  \begin{itemize}
   \item $M_0, N_0$ are positive integers coprime to $p$ with $M_0^2 \mid N_0$, and $\chi_0$ is a Dirichlet character of conductor $M_0$ (valued in $L$).
   \item $M_2, N_2$ are positive integers coprime to $p$ with $M_2 \mid N_2$, and $\chi_2$ is a Dirichlet character of conductor $M_2$ (valued in $L$).
  \end{itemize}
  We will consider automorphic representations $\pi$ of $G$ with conductor $N_0$ and character $\widehat{\chi}_0$ up to twists by norm, and similarly $\sigma$ of $\GL_2$ with conductor $N_2$ and character $\widehat{\chi}_2$ up to twists by norm.\footnote{This numbering of the parameters comes from the fact that the zeta-integral computations of \cite{LPSZ1} are simpler to write down if the Eisenstein series lives on the first factor of $H$.}

  Let $S$ denote the set of primes dividing $N_0 N_2$. By \emph{tame test data} we shall mean a pair $\gamma_S = (\gamma_{0, S}, \Phi_S)$, where:
  \begin{itemize}
  \item $\gamma_{0, S} \in G(\QQ_S)$, where $\QQ_S = \prod_{\ell \in S} \QQ_\ell$;
  \item $\Phi_S \in C^\infty_c(\QQ_S^2, L)$, lying in the $\left(\widehat{\chi}_0 \widehat{\chi}_2\right)^{-1}$-eigenspace for $\ZZ_S^\times$.
  \end{itemize}
  We let $K_S$ be the quasi-paramodular subgroup of $G(\QQ_S)$ of level $(N_0, M_0)$; and we let $\widehat{K}_S$ be some open compact subgroup of $G(\QQ_S)$ such that:
  \begin{itemize}
   \item $\widehat{K}_S \subseteq \gamma_{0, S} K_S \gamma_{0, S}^{-1}$,
   \item the projection of $\widehat{K}_S \cap H$ to the first factor of $H$ acts trivially on $\Phi_S$,
   \item  the projection of $\widehat{K}_S \cap H$ to the second factor of $H$ is contained in $\{ \stbt \star\star0 1 \bmod N_2\}$.
  \end{itemize}
  We define $K^p$ and $\widehat{K}^p$ to be the products of $K_S$ and $\widehat{K}_S$ with $G(\Af^{pS})$, and $\Phi^{(p)} = \Phi_S \cdot \operatorname{ch}\left((\widehat{\ZZ}^{S \cup \{p\}})^2\right)$.

 \subsection{The correction term $Z_S$}

  Let $\pi$ and $\sigma$ be cohomological cuspidal automorphic representations of $G$ and of $\GL_2$, both defined over some number field $E$ contained in the $p$-adic field $L$, and both globally generic and unramified outside $S$. We normalise so these are cohomological with weights $(r_1, r_2; r_1 + r_2)$ and $(t_2; t_2)$ respectively, for some integers $r_1, r_2, t_2$; and we let $\Pi$ and $\Sigma$ be the unitary twists of $\pi$ and $\sigma$ respectively, so that
  \[ L(\Pi \times \Sigma, s) = L(\pi \times \sigma, s + \tfrac{r_1 + r_2 + t_2}{2}).\]

  \begin{definition}
   For $W_0 \in \cW(\pi)_E$, $W_2 \in \cW(\sigma)_E$, and $\Phi \in \cS(\QQ_S^2, E)$, we consider the zeta-integral
   \[ Z(W_0, \Phi, W_2; s) = \int_{(Z_G N_H \backslash H)(\QQ_S)} W_0(h) f^{\Phi}(h_1; \omega_{\pi}\omega_\sigma, s) W_2(h_2) \, \mathrm{d}h. \]
  \end{definition}

  We shall set
  \[
   Z_S(\pi \times \sigma, \gamma_S; s) =
    \frac{Z(\gamma_{0, S} \cdot W_0^{\new}, \Phi_S, W_2^\new; s)}{G(\chi_2^{-1})\prod_{\ell \in S} L(\pi_\ell \times \sigma_\ell, s)},
  \]
  and
  \[ Z_S(\pi \times \sigma, \gamma_S) = Z_S(\pi \times \sigma, \gamma_S; 1 + \tfrac{t_1}{2})\]
  where $t_1 = r_1 -r_2-2-t_2$ as usual. Here $G(\chi) = \sum_{a \bmod N_\chi} \chi(a) \exp(2\pi i a / N_\chi)$ is the Gauss sum of the character $\chi$. One can check that this is a product of polynomials in the variables $\ell^{\pm s}$, for $\ell \in S$, with coefficients in $E$.

  \begin{proposition}
   For any given $\pi, \sigma$, one can choose $\gamma_S$ such that $Z_S(\pi \times \sigma, \gamma_S; s) \ne 0$.
  \end{proposition}

  \begin{proof}
   This follows from the definition of the $L$-factor as a GCD of local zeta-integrals.
  \end{proof}

 \subsection{P-adic families for $G$}

  Let $U \subset \cW^2$ be an open affinoid disc; and let $\br_1$, $\br_2: \Zp^\times \to \cO(U)^\times$ be the universal characters associated to the two factors of $\cW^2$. Let $\nu_U$ be the character $(\br_1, \br_2; \br_1+\br_2)$ of $T(\Zp)$.

  The theory of \cite{boxerpilloni20} shows that there exists a rigid space $\cE \xrightarrow{\kappa} \cW^2$, with a map $\mathbb{T}^- \to \cO(\cE)$ (the eigenvariety for $G$), and graded coherent sheaves $H^k(\cM^{\bullet,-,\fs}_{\cusp, w_j})$ on $\cE$ for $0 \le j, k \le 3$, whose pushforward to any affinoid $U\subset \cW^2$ as above is $H^k_{w_j,\an}(K^p, \nu_U, \cusp)^{(-,\fs)}$. By construction, the points of $\cE$ biject with systems of $\mathbb{T}^-$-eigenvalues appearing in one of these modules.

  \begin{definition}
  By a \emph{family of automorphic representations} $\upi$ over $U$ (of tame level $N_0$ and character $\chi_0$), we mean the data of a finite flat covering $\tilde{U} \to U$, and a homomorphism $\tilde{U} \to \cE$ lifting the inclusion $U \into \cW$, such that the following conditions hold:
  \begin{itemize}
  \item $\tilde{U}$ is 2-dimensional and smooth;
  \item the restriction of the sheaf $H^k(\cM^{\bullet,-,\fs}_{\cusp, w_j})$ to $\tilde{U}$ is zero if $k \ne 3-j$, and the sheaves $S^k(\upi) = H^k(\cM^{\bullet,-,\fs}_{\cusp, w_{3-k}})$ are either free over $\cO(\tilde{U})$ of rank 1 for all $k$ (a general-type family), or free of rank 1 for $k = 1,2$ and zero for $k = 0, 3$ (a Yoshida-type family);
  \item the centre of $G(\Af^p)$ acts on the modules $S^k(\upi)$ by the character $|\cdot|^{-(\br_1 + \br_2)} \widehat{\chi}_0$.
  \end{itemize}
  \end{definition}

  Such a family determines a $\cO(\tilde{U})$-valued character $\lambda_{\upi}^-$ of $\mathbb{T}^-$, which is the system of eigenvalues by which $\mathbb{T}$ acts on the modules $H^k(\cM^{\bullet,-,\fs}_{\cusp, w_j})$; conversely, the character $\lambda_{\upi}^-$ and ring extension $\cO(\tilde U)$ of $\cO(U)$ uniquely determine $\upi$.

  \begin{definition}
  We say a point $P \in \tilde{U}(L)$ is ``good for $\upi$'' if the following conditions hold:
  \begin{itemize}
   \item the weight of $P$ is $(r_1, r_2) \in U \cap \ZZ^2$ with $r_1 \ge r_2 \ge -1$;
   \item the specialisation at $P$ of the system of eigenvalues $\lambda^-_{\upi}$ is the character of $\mathbb{T}^-$ assocated to a $p$-stabilised automorphic representation $\pi_P$, which is cuspidal, globally generic, and has conductor $N_0$ and character $\chi_0$;
   \item the fibre of $S^2(\upi)$ at $P$ maps isomorphically to the $\pi_P$-eigenspace in the classical $H^2(K^p, \kappa_1(\nu), \cusp)$; in particular, this eigenspace is 1-dimensional.
  \end{itemize}
  \end{definition}

  \begin{remark}
   Note that we do not suppose that the $\pi_P$ \emph{generalised} eigenspace be 1-dimensional, and this will not hold when $\tilde{U} \to U$ is ramified at $P$.
  \end{remark}

  By the classicity theorems for higher Coleman theory recalled above, given a family $\upi$, all specialisations of integer weight $(r_1, r_2)$ with $r_1 - r_2$ and $r_2$ sufficiently large relative to the slope of $\upi$ will be good; and if $\upi$ is ordinary, it suffices to assume that $r_1-r_2 \ge 3$ and $r_2 \ge 0$.

  We shall choose a basis $\uet$ of $S^2(\upi)$. Since the spaces of higher Coleman theory (of varying levels) have an action of $G(\Af^p)$, we can make sense of $\gamma_{0, S} \cdot \uet$ as a family of classes at tame level $\widehat{K}^p$, which is still an eigenfamily for the Hecke operators away from $S$.

 \subsection{Families for $\GL_2$}

  Similarly, we choose a disc $U' \subset \cW$, a finite flat covering $\tilde{U}'\to U'$ with $\tilde{U}'$ smooth, and a finite-slope overconvergent $p$-adic family of modular eigenforms $\cG$ over $\tilde{U}'$ (of weight $\mathbf{t_2} + 2$ where $\mathbf{t_2}$ is the universal character associated to $U'$). We suppose that this family is new away from $p$ of tame level $N_2$, and nebentype character $\chi_2$.

  We say a point $Q \in \tilde{U}'$ is ``good for $\cG$'' if it lies above an integer $t \in U' \cap \ZZ_{\ge -1}$, and the specialisation of $\cG$ at $Q$, which is \emph{a priori} an overconvergent form of weight $t + 2$, is in fact a classical form. (This is automatic if $t$ is sufficiently large compared to the slope of $\cG$.) We write $\sigma_t$ for the corresponding automorphic representation (normalised to have central character $|\cdot|^{-t} \widehat{\chi}_2$); and we formally write $\usi$ for the collection of the $\sigma_t$ for varying $t$.

  A mildly irritating detail is that if $\cG$ is normalised to have $a_1(\cG) = 1$, and $t$ is a good specialisation, then $\cG_t$ has $q$-expansion coefficients in some number field $E$; but the modular form $\cG_t$ is not defined over $E$ as a coherent cohomology class, since the cusp $\infty$ on $X_1(N)$ is not defined over $\QQ$ (with our conventions). However, the class $G(\chi_2^{-1})\cG_t$ is $E$-rational. We write $S^0(\sigma_t, E)$ for the $E$-vector space spanned by this form, and similarly $S^0(\usi)$ for the $\cO(\tilde{U}')$-module of overconvergent cusp forms generated by $G(\chi_2^{-1}) \cG$.

  \begin{remark}
   Note that by definition $S^0(\usi)$ is free of rank 1, and its fibre at any good specialisation is in the image of the classical $H^0$ (because of the $q$-expansion principle for $p$-adic modular forms). Hence we do not need any auxiliary hypotheses about local freeness of sheaves.
  \end{remark}

 \subsection{Deforming eigenforms}

  Conversely, we say a classical ($p$-stabilised) automorphic representation $\pi$, of some weight $\nu$, is \emph{deformable} if we can find a disc $U$ containing $\nu$, a family $\upi$ over some covering $\tilde{U} / U$, and some $Q \in \tilde{U}$ above $\nu$, such that $Q$ is good for $\upi$ and the specialisation there is $\pi$. The arguments of \cref{sect:families} show that any generic $\pi$ of cohomological weight, with a regular $p$-stabilisation of sufficiently small slope, will be deformable in the above sense (and we may suppose that $\tilde{U}=U$); again, if $\pi$ is ordinary, it suffices to suppose that $r_1 - r_2 \ge 3$ and $r_2 \ge 0$.

  For $\GL_2$ we are in much better shape (partly because $\GL_2$ is better understood than $\GSp_4$, and partly because our definition of ``family'' is less restrictive): any classical $p$-stabilised newform of integer weight and Iwahori level at $p$ will be deformable, even in the worst-case scenario of non-$p$-regular weight 1 forms, since we may take $\tilde{U}'$ to be a neighbourhood of $\sigma$ in the normalisation of the eigencurve. Moreover, if $\sigma$ is ordinary and has weight $\ge 2$, we may suppose $\tilde{U}' = U'$.

  \begin{remark}
   We also expect that there exist interesting examples of deformable $\pi$ for $G$ which do not satisfy these stringent conditions. It seems likely that the extra generality of a finite flat covering of weight space will be genuinely necessary, at least in the non-regular-weight case $r_2 = -1$. However, for simplicity of notation we shall assume $\tilde{U} = U$ and $\tilde{U}' = U'$ henceforth; extending these arguments to the general case is straightforward and we leave this to the reader.
  \end{remark}

  \subsubsection*{Families over $U \times U'$} Let $A = \cO(U \times U')$. We have two canonical $A$-valued characters of $T(\Zp)$: the canonical character $\nu_A = (\br_1, \br_2; \br_1 + \br_2)$, and the character $\tau_A = (\bt_1, \bt_2; \br_1 + \br_2)$ defined as follows: $\bt_2$ is the canonical character of $U'$ as above, and $\bt_1 = \br_1 - \br_2 - 2 - \bt_2$ and the action of the centre are determined by the requirement that $\nu_A$ and $\tau_A$ be ``compatible'' in the sense of \cref{def:compat}. Then we can consider
  \[ \cE^{\Phi^{(p)}}(0, \bt_1 + 1) \boxtimes G(\chi_2^{-1}) \cG^{[p]} \in H^0_{\id, \an}(\cS_{H, \Iw}(p^2), \tau_{A})^{+, \dag},\]
  where the tame level is taken to be $H \cap \widehat{K}^p$.

  \begin{definition}
   We let $\cL_{p, \gamma_S}(\upi \times \usi; \uet)$ denote the element of $A$ defined by
   \[ \left\langle \hat\iota^*\left(\gamma_{0, S} \cdot \uet\right), \cE^{\Phi^{(p)}}(0, \bt_1 + 1) \boxtimes G(\chi_2^{-1}) \cG^{[p]}\right\rangle.\]
  \end{definition}

  (The product denotes the Serre duality pairing at level $\widehat{K}^p \cap H$, normalised by a factor $\operatorname{vol}(\widehat{K}^p \cap H)$ in order to make it independent of the choice of $\widehat{K}_p$.)
  \begin{definition} \
   \begin{itemize}
    \item We say a point $(P, Q)$ of $U \times U'$ is \emph{good} if $P = (r_1, r_2)$ and $Q = (t_2)$ are integer points, with $P$ good for $\upi$ and $Q$ good for $\usi$.

    \item We say $(P, Q)$ is \emph{good critical} if we also have $t_2 \le r_1 -r_2 - 1$ (i.e.~the specialisation $t_1$ of $\bt_1$ at $(P, Q)$ is $\ge -1$).

    \item If instead we have $r_1 - r_2 \le t_2 \le r_1$, we say $P$ is \emph{good geometric}.
   \end{itemize}
  \end{definition}

  One checks easily that any integer point $(r_1, r_2, t_2)$ is the limit of a sequence of good geometric (or good critical) points, so if we exclude the pathological case when $(U \times U') \cap \ZZ^3$ is empty, then the sets of good critical points and of good geometric points are both Zariski-dense in $U \times U'$.

 \subsection{Values in the critical range}

  \begin{definition}
   For $(P, Q) = (r_1, r_2, t_2) \in U \times U'$ a good critical point, we define a degree 8 Euler factor
   \[
    \cE_p(\pi_P \times \sigma_Q) = \left(1 - \tfrac{p^{r_1 + 1}}{\alpha \fa}\right)\dots \left( 1- \tfrac{p^{r_1 + 1}}{\beta \fb}\right) \left(1 - \tfrac{\gamma \fa}{p^{r_1 + 2}}\right)\dots\left(1 - \tfrac{\delta \fb}{p^{r_1 + 2}}\right).
   \]
   where $\alpha, \dots, \delta$ are the Hecke parameters of $\pi_P$, and $\fa,\fb$ the Hecke parameters of $\sigma_Q$ (so that $\fa\fb = p^{t_2 + 1} \chi_2(p)$).
  \end{definition}

  \begin{proposition}
   If $\pi_P$ is ordinary, then $\cE_p(\pi_P\times \sigma_Q) \ne 0$.
  \end{proposition}

  \begin{proof}
   This follows by a (somewhat tedious) explicit check from the bounds on the valuations of the Hecke parameters.
  \end{proof}

  \begin{theorem}
   The $p$-adic $L$-function $\cL_{p, \gamma_S}(\upi \times \usi, \uet)$ has the following interpolation property: if $(P, Q)$ is good critical, then
   \[ \frac{\cL_{p, \gamma_S}(\upi \times \usi, \uet)(P, Q)}{\Omega_p(\pi_P, \eta_P)}
   =Z_S(\pi_{P} \times \sigma_{Q}, \gamma_S) \cdot \cE_p(\pi_P \times \sigma_Q) \cdot
   \frac{G(\chi_2^{-1})^2 \Lambda\left(\Pi_P \times \Sigma_Q, 1 + \tfrac{t_1}{2}\right)}{\Omega_\infty(\pi_P, \eta_P)}, \]
   with both sides lying in the field of rationality of $\pi_P\times \sigma_Q$.

   Here $\Pi_P$ and $\Sigma_Q$ are the (unitary) automorphic representations generated by the specialisations of $\uet$ and $\cG$ at $P$; and $\Lambda(\Pi_P\times \Sigma_Q, s)$ denotes the $L$-function of these automorphic representations, with its archimedean $\Gamma$-factors included.
  \end{theorem}

  \begin{remark}
   Note that $s = 1 + \tfrac{t_1}{2}$ is the upper endpoint of the interval of critical values (in the sense of Deligne) for the degree 8 $L$-function $L\left(\Pi_P \times \Sigma_Q, s\right)$. This critical interval is symmetric about $s = \tfrac{1}{2}$, so unless $t_2 = r_1 - r_2 - 1$ (so that $t_1 = -1$), there are other critical values which we do not see by this method.

   We optimistically hope that there should be a $p$-adic $L$-function on the 4-dimensional space $U \times U' \times \cW$ which interpolates the full range of critical values, and that both the above $p$-adic $L$-function on $U \times U'$, and the 2-variable $p$-adic $L$-function on $U' \times \cW$ (for fixed $\pi$) considered in \cite[\S 5]{LZ20b-regulator}, should be ``slices'' of this more general construction. However, this seems beyond reach with our present methods.
  \end{remark}

  \begin{proof}
   By construction, we have
   \[ \cL_{p, \gamma_S}(\upi \times \usi; \uet)(P, Q) = G(\chi_2^{-1}) \left\langle \hat\iota^*\left(\gamma_{0, S} \cdot \eta_P\right), \cE^{\Phi^{(p)}}(0, t_1 + 1) \boxtimes \cG_P^{[p]}\right\rangle. \]
   This expands as the product of $G(\chi_2^{-1})\Lambda\left(\Pi_P \times \Sigma_Q, 1 + \tfrac{t_1}{2}\right)$ and a product of normalised local zeta-integrals, exactly as in \cite{LPSZ1}. The local integrals away from $pS$ are all 1. The local zeta-integral at $p$ is evaluated in \cite{LZ21-zeta2}, and gives the Euler factor $\cE_p(-)$. The product of zeta-integrals at the bad primes is by definition $G(\chi_2^{-1}) Z_S(\dots)$ and the result follows.
  \end{proof}

 \subsection{Values in the geometric range}

  Suppose $(P, Q)\in U \times U'$ is a point in the good geometric range; and let us set $t_1' = -2-t_1 = t_2 - r_1 + r_2$. Then the ``geometric'' condition implies that $0 \le t_1' \le r_2$, and the quadruple $(r_1, r_2, t_1', t_2)$ satisfies the branching law for algebraic representations defined in \cite[Proposition 6.4]{LPSZ1}, which is the condition needed to define motivic cohomology classes associated to $\pi_P \otimes \sigma_Q$, using the pushforward of a $\GL_2$ Eisenstein class of weight $t_1'$ (see \cite{HJS20}).

  \begin{remark}
   Note that this Euler system class lands in the Galois representation $V_p(\pi \times \sigma)^*(-1-r_1)$, and corresponds to the complex $L$-function $L(\pi \times \sigma, s)$ at $s = -\tfrac{t_1'}{2} = 1 + \tfrac{t_1}{2}$; but this is no longer a critical value, and the Archimedean $\Gamma$-factors force the $L$-function to vanish here to degree exactly one (except in some exceptional cases when $t_1' = 0$ and $\pi$ is a Yoshida lift, when it can happen that the completed $L$-function has a simple pole at $s = 0, 1$).

   The $L$-values having this property are an interval (disjoint from the critical interval, if any) and the value $s = -\tfrac{t_1'}{2}$ is the \emph{upper} end of this interval. So our restriction to using only overconvergent, rather than nearly-overconvergent, Eisenstein series pegs us to the the upper endpoint of the critical interval when $P$ is critical, and to the upper endpoint of the geometric interval when $P$ is geometric.
  \end{remark}

  In \cite[\S 4]{LZ20b-regulator}, we defined an object $\operatorname{Per}_{\eta}(\pi_P \times \sigma_Q)$ associated to $\pi_P \times \sigma_Q$, the choice of twist $t_1'$, and the basis vector $\eta_P \in S^2(\pi_P, L)$. This was an $H(\Af^p)$-equivariant map $\mathcal{T}^p \to L$, where $\mathcal{T}^p = \cW(\pi_{P, \mathrm{f}}^p) \otimes C^\infty_c( (\Af^p)^2) \otimes \cW(\sigma_{Q, \mathrm{f}}^p)$. Our choice of $\gamma_S$ defines a choice of vector
  \[ (\gamma_{0, S} W^{\new}_{\pi_P}) \otimes \Phi_S \otimes W^{\new}_{\sigma_Q} \in \mathcal{T}^p \]
  and we write $\operatorname{Per}_{\eta}(\pi \times \sigma, \gamma_S) \in L$ for the value of $\operatorname{Per}_{\eta}(\pi \times \sigma)$ on this vector.

  \begin{remark}
   If $t_1 \ne 0$, then one can check that the space of $H(\Af^p)$-equivariant maps in which $\operatorname{Per}_{\eta}(\pi_P \times \sigma_Q)$ lies is in fact 1-dimensional and spanned by the product of zeta integrals used to define $Z_S(\dots)$. It follows that there is a quantity $\operatorname{Per}_{\eta}(\pi_P \times \sigma_Q)^{\mathrm{univ}} \in L$ such that for all $\gamma_S$ we have
   \[  \operatorname{Per}_{\eta}(\pi_P \times \sigma_Q, \gamma_S) = Z_S(\pi_{P} \times \sigma_{Q}, \gamma_S) \operatorname{Per}_{\eta}(\pi_P \times \sigma_Q)^{\mathrm{univ}}. \]
   Similar results also hold for $t_1 = 0$ under some mild additional conditions on $\pi_P$ and $\sigma_Q$; compare Theorem 6.6.2 of \cite{LZ20} in the $\GSp_4$ case. However, we do not need this for the proof of our main theorem, so we shall not pursue it further here.
  \end{remark}

  \begin{proposition}
   \label{prop:coherentperiod}
   We have
   \[
    \cL_{p, \gamma_S}(\upi \times \usi; \uet)(P, Q) =  \operatorname{Per}_{\eta_P}(\pi_P \times \sigma_Q, \gamma_S).
   \]
  \end{proposition}

  \begin{proof}
   By construction, we have
   \[
    \operatorname{Per}_{\eta_P}(\pi_P \times \sigma_Q, \gamma_S) = \left\langle \iota^*_{\Kl}\left(\gamma_{0, S}\cdot \eta_{P, \Kl}\right), \theta^{-{t_1'}}\left(F^{(t_1' + 2)}_{\Phi^p \Phi_{\mathrm{dep}}}\right) \boxtimes G(\chi_2^{-1}) \cG^{[p]}\right\rangle, \]
   where $\iota_{\Kl}$ is an embedding of Shimura varieties at Klingen level. The term on the right-hand side is exactly the specialisation at $P$ of our family of $p$-adic modular forms for $H$. From the zeta-integral computations of \cite{LZ21-zeta2}, we may replace $\iota^*_{\Kl}\left(\eta_{\Kl}\right)$ with $\hat\iota^*\left(\eta_{\Iw}\right)$ without changing the value of the pairing.
  \end{proof}

\section{Families of cohomology classes}

 We persist with the notation and assumptions of the previous section. We also suppose that the family $\upi$ is not of Yoshida type, so that for each classical specialisation $P$, the $\lambda_{P}$-eigenspace in \'etale cohomology of $S_{G, K^p\Iw(p)}$ is 4-dimensional. We suppose furthermore that $\upi$ and the $\GL_2$ family $\usi$ are ordinary at $p$.

 \subsection{Galois representations}

  Associated with the family $\upi$ we have a family of Galois representations $V(\upi)$, which is a rank 4 $\cO(U)$-module with an action of $\Gal(\overline{\QQ}/\QQ)$, unramified outside $pN_0$ and satisfying $\operatorname{tr}(\operatorname{Frob}_\ell^{-1} | V(\upi)) = \lambda(T_{1,\ell})$ for $\ell \nmid pN_0$.

  The existence of this family is a consequence of the results of \cite{tilouineurban99}, who also give a canonical realisation of the dual representation $V(\upi)^*$ as a localisation of the module
  \[ e'_B \cdot \varprojlim_t H^3_{\text{\textup{\'et}}, c}\left(\cS_{G, K^pK_{p, t}, \overline{\QQ}}, \Zp\right) \otimes_{\Zp[[\Zp^{\times 2}]]} \cO(U), \]
  where $K_{p, t}$ is some family of subgroups of $G(\Zp)$ and $e^-_B$ is the ordinary projector associated to $\cU'_B$. Similarly, there is a 2-dimensional family of Galois representations over $U'$ associated to $\usi$.

  \begin{remark}
   If the family $\upi$ has a classical specialisation whose weight is sufficiently regular, but small relative to $p$ (and some additional hypotheses hold regarding the image of the residual Galois representation), then the results of \cite{mokranetilouine02} and \cite{rockwood-control} imply that $V(\upi)$ is free of rank 4 over $\cO(U)$.

   Without this condition, we can only deduce that $V(\upi)$ is locally free in a neighbourhood of each good \emph{cohomological} weight, but not necessarily elsewhere. One can work around this by replacing $V(\upi)^*$ with its double dual (reflexive hull), which does not change its specialisations in cohomological weights.
  \end{remark}

  \begin{definition}
   We set
   \[\mathbb{V}^* = V(\upi)^* \times V(\usi)^*(-1-\br_1),\]
   which is an 8-dimensional family of Galois representations over $U \times U'$.
  \end{definition}

 \subsection{Ordinary filtrations at $p$}

  The Galois representation $V(\upi)$ has a decreasing filtration by $\cO(U)$-submodules stable under $\Gal(\overline{\QQ}_p/\Qp)$ (via results of Urban \cite{urban05}; see \cite[Theorem 17.3.1]{LZ20} for the formulation we use). We write $\cF^i V(\upi)$ for the codimension $i$ subspace, and similarly for its dual $V(\upi)^*$. Note that $\Gr^0 V(\upi)^*$ is unramified, with arithmetic Frobenius acting as the $U_{\Sieg}$-eigenvalue. Abusing notation slightly\footnote{What we really mean is that $\Gr^1 V(\upi)^*$ is isomorphic to the tensor product of $\chi_{\mathrm{cyc}}^{(1 + \br_2)}$ and an unramified character.}, we may say that $\Gr^1 V(\upi)^*$ has ``Hodge--Tate weight $1 + \br_2$''.

  Similarly, there is a 2-step filtration of $V(\usi)^*$, with $\Gr^1V(\usi)^* = \cF^1 V(\usi)^*$ having Hodge--Tate weight $1 + \bt_2$.

  \begin{definition}
   We set
   \[\mathbb{V}^* = V(\upi)^* \times V(\usi)^*(-1-\br_1);\]
   and we let
   \[ \cF^{(f)} V(\upi \times \usi)^* = \cF^2 V(\upi) \otimes V(\usi)^*,\]
   and
   \[ \cF^{(e)} V(\upi \times \usi)^* = \left(\cF^2 V(\upi)^* \otimes V(\usi)^*\right) + \left(\cF^1 V(\upi)^* \otimes \cF^1 V(\usi)^*\right). \]
   For a good weight $(P, Q)$ we write $\mathbb{V}_{P, Q}^*$ for the specialisation of $\mathbb{V}^*$ at $(P, Q)$, so $\mathbb{V}_{P, Q}^* = V(\pi_P)^* \otimes V(\sigma_Q)^*(-1-r_1)$ if $P = (r_1, r_2)$.
  \end{definition}

  (For the significance of the labels (e) and (f), see Figure 2 of \cite{LZvista}.) Thus $\cF^{(e)}$ has rank 5, $\cF^{(f)}$ has rank 4, and the quotient $\Gr^{(e/f)} \cong \left(\Gr^1 V(\pi)^*\right)\otimes\left(\cF^1 V(\usi)^*\right)(-1-\br_1)$ has Hodge--Tate weight $\bt_1' = -2-\bt_1$.

  \begin{remark}
   Note that
   \[ \cE_p(\pi_P\times \sigma_Q) = \det \left(1 - \varphi: \Dcris(\cF^{(f)} \mathbb{V}_P^*)\right) \cdot \det\left( 1 - p^{-1}\varphi^{-1}: \Dcris \left(\mathbb{V}_P^* / \cF^{(f)}\right) \right).\qedhere\]
  \end{remark}

 \subsection{P-adic periods}

  The representations $\Gr^1V(\upi)^*(-1-\mathbf{r_2})$ and $\Gr^1 V(\usi)^*(-1-\bt_2)$ are unramified, and hence crystalline as $\cO(U)$ (resp.~$\cO(U')$)-linear representations. Since $\Dcris(\Qp(1))$ is canonically $\Qp$, we can therefore define $\Dcris(\Gr^{(e/f)} \mathbb{V}^*)$ to be an alias for the rank 1 $\cO(U \times U')$-module
  \[
   \Dcris\left(\Gr^1 V(\pi)^*(-1-\mathbf{r_2})\right) \htimes \Dcris\left( \Gr^1 V(\usi)^*(-1-\bt_2)\right).
  \]
  As in \cite[\S 8.2]{KLZ17}, we can define a Coleman/Perrin-Riou big logarithm map for $\Gr^{(e/f)} \mathbb{V}^*$, which is a morphism of $\cO(U \times U')$-modules
  \[ \cL^{\mathrm{PR}}: H^1(\Qp, \Gr^{(e/f)} \mathbb{V}^*) \to \Dcris(\Gr^{(e/f)} \mathbb{V}^*). \]
  By construction, for good geometric weights $P$, this specialises to the Bloch--Kato logarithm map, up to an Euler factor; and for good critical weights it specialises to the Bloch--Kato dual exponential.

 \subsection{P-adic Eichler--Shimura isomorphisms}

  Let $P$ be a good weight. Then the Faltings--Tsuji comparison isomorphism of $p$-adic Hodge theory gives an identification between $\Dcris(V(\pi_P))$ and the $\pi_P$-eigenspace in de Rham cohomology (compatibly with the Hodge filtration); and the graded pieces of this filtration are identified with the coherent cohomology groups $S^i(\pi_P, L)$.

  Since the Hodge and Newton filtrations on $\Dcris$ must be complementary to each other (by weak admissibility), we deduce that there is an \emph{Eichler--Shimura} isomorphism
  \[
   \ES^2_{\pi_P}: S^2(\pi_P, L)\cong \Gr^{(r_2 + 1)}_{\mathrm{Hdg}} \Dcris(V(\pi_P))  \cong \Dcris(\Gr^2 V(\pi_P)).
  \]
  Concretely, the isomorphism is given by mapping an element in $\Gr^{(r_2 + 1)}_{\mathrm{Hdg}} \Dcris(V(\pi_P))$ to its unique lifting to $\Fil^{(r_2 + 1)}_{\mathrm{Hdg}}\Dcris(V(\pi_P)) \cap \ker( (\varphi - \alpha_P)(\varphi - \beta_P))$.

  \begin{remark}
   More generally, we have isomorphisms $\ES^i: S^i(\pi_P, L)\cong \Dcris(\Gr^i V(\pi_P))$ for each $0 \le i \le 3$, where $S^i(\pi_P, L)$ is the $\pi_P$-eigenspace in coherent $H^i$.

   We caution the reader that although the source and target of $\ES^i_{\pi_P}$ are the specialisations at $P$ of rank-one $\cO(U)$-modules, it is \textbf{by no means obvious} that the isomorphisms $\ES^i_{\pi_P}$ for varying $P$ are the specialisations of a single $\cO(U)$-module isomorphism ``$\ES^i_{\upi}$''. We shall establish (a slightly weakened form of) this below, under some additional hypotheses, as a by-product of our main Euler system argument.

   It would be very interesting to have a direct construction of the maps $\ES^i_{\upi}$ by methods of arithmetic geometry. For $i = 0$ (corresponding to classical holomorphic Siegel modular forms) this has been achieved in the recent preprint \cite{diao-rosso-wu21}. One can also obtain $\ES^3_{\upi}$ from this via Serre duality; but it seems to be more difficult to construct the ``intermediate'' filtration steps $i = 1, 2$.
  \end{remark}

  \subsubsection{Analogue for $\GL_2$} Similarly, for $\GL_2$ we have an isomorphism
  \[ \ES^0_{\sigma_Q}: S^0(\sigma_Q, L) \cong \Dcris(\Gr^0 V(\sigma_Q)). \]
  In this setting the existence of comparison isomorphisms in families is known:

  \begin{theorem}[Ohta, Kings--Loeffler--Zerbes]
   There exists an isomorphism of $\cO(U')$-modules
   \[ \ES^0_{\usi}: S^0(\upi) \cong \Dcris(\Gr^0 V(\usi))\]
   interpolating the isomorphisms $\ES^0_{\sigma_Q}$ for varying $P$, where $\cS^0(\upi)$ is the $\cO(U')$-module spanned by $\underline{\omega} = G(\chi_2^{-1}) \cdot \cG$.
  \end{theorem}
  \begin{proof}
   This is a restatement of \cite[Proposition 10.1.1(1)]{KLZ17}, where it is derived from results of Ohta \cite{ohta00}. For an alternative derivation applying to possibly non-ordinary Coleman families, see \cite{andreattaiovitastevens,loefflerzerbes16}.
  \end{proof}

 \subsection{Euler system classes}
  Let us suppose that the character $\chi_0 \chi_2$ is non-trivial (this allows us to get rid of a ``smoothing factor'' $c$ appearing in the Euler system constructions). Then, associated to the data $\gamma_S$, we also have a family of cohomology classes
  \[ \mathbf{z}_{m}(\upi \times \usi, \gamma_S) \in H^1(\QQ(\mu_{m}), \mathbb{V}^*), \]
  for all square-free integers coprime to some finite set $T \supseteq S \cup \{p\}$.
  By construction, the image of $ \mathbf{z}_{m}(\upi \times \usi, \gamma_S)$ under localisation at $p$ lands in the image of the (injective) map from the cohomology of $\cF^{(e)} \mathbb{V}^*$. So we may make sense of
  \[  \cL^{\mathrm{PR}}\left( \mathbf{z}_{m}(\upi \times \usi, \gamma_S) \right) \in \Dcris(\Gr^{(e/f)} \mathbb{V}^*). \]
  We denote its image under specialisation at $(P,Q)$ by $ \cL^{\mathrm{PR}}\left( \mathbf{z}_{m}(\upi \times \usi, \gamma_S) \right)(P,Q)$.
  Combining \cref{prop:coherentperiod} with the main result of \cite{LZ20b-regulator}, which relates the periods $\operatorname{Per}_\eta(\dots)$ to the Euler system classes, we have the following result:

  \begin{theorem}\label{prop:geomreg}
   For each $P$ in the good geometric range, we have
   \[ \left\langle \cL\left( \mathbf{z}_{1}(\upi \times \usi, \gamma_S) \right)(P,Q), \mathrm{ES}^2_{\pi_P}(\eta_P) \otimes \mathrm{ES}^0_{\sigma_Q}(\omega_P)\right\rangle = \cL_{p,\gamma_S}(\upi \times \usi; \uet)(P, Q). \]
  \end{theorem}

 \subsection{Reciprocity laws and meromorphic Eichler--Shimura}

  \begin{definition}
   Let $\fS(\upi; \usi)$ denote the set of points $P = (r_1, r_2) \in U \cap \ZZ^2$ which are good for $\upi$, and satisfy the following condition: there exists some $t_2 \in U' \cap \ZZ_{\ge 0}$, and some local data $\gamma_S$, such that $(P, Q) = (r_1, r_2, t_2)$ is good geometric and $\cL_{p, \gamma_S}(\upi \times \usi;\uet)$ is non-vanishing at $(P, Q)$.
  \end{definition}

  \begin{lemma}
   Let $\usi, \usi'$ be two Hida families satisfying our running hypotheses (possibly of different tame levels and characters). Then the set $\fS(\upi,\usi) \cap \fS(\upi,\usi')$ is Zariski-dense. In particular, $\fS(\upi,\usi)$ is itself Zariski-dense.
  \end{lemma}

  \begin{proof}
   We first note that there exists $\gamma_S$ for which $\cL_{p, \gamma_S}(\upi \times \usi;\uet)$ is not identically zero. To see this, we choose some good \emph{critical} point $(P, Q)$ having $t_1 = r_1 - r_2 - t_2 - 2 \ge 0$, so that $\Lambda(\pi_P \times \sigma_Q, 1 + \tfrac{t_1}{2})$ lies outside the strip $0 < \Re(s) < 1$ and hence cannot vanish. We can then choose $\gamma_S$ such that $Z_S(\pi_P \times \sigma_Q, \gamma_S) \ne 0$ (which is always possible). Thus $\cL_{p, \gamma_S}(\upi \times \usi)$ is non-vanishing at $(P, Q)$, and hence generically non-vanishing on $U \times U'$.

   Repeating the construction, we can find local data $\gamma_{S}'$ for $\upi \times \usi'$ such that $\cL_{p, \gamma_S'}(\upi \times \usi')$ is generically non-vanishing. So there is an open subset $V \subset U$ such that for all $v \in V$, neither $\cL_{p, \gamma_S}(\upi \times \usi)$ nor $\cL_{p, \gamma_S'}(\upi \times \usi')$ vanishes identically along $\{ v \} \times U'$.

   Since $V$ is open, it must contain some $(r_1, r_2) \in V \cap \ZZ^2$; and we can therefore find an integer $t$ such that both $p$-adic $L$-functions are non-vanishing at $P = (r_1, r_2, t)$. We consider the sequence of weights $P_k = (r_1 + 3(p-1)p^{k}, r_2 + (p-1)p^{k}, t_2 + 2(p-1)p^k)$ for $k \to \infty$. For all but finitely many $k$ the weight $P_k$ will be good geometric, and $P_k$ tends to $P$, so $\cL_{p, \gamma_S}(P_k) \ne 0$ for sufficiently large $k$. Thus the projection of $P_k$ to $U$ lies in $\fS(\upi \times \usi)$, and also in $\fS(\upi \times \usi')$. It follows that $(r_1, r_2)$ is a limit point of $\fS(\upi \times \usi) \cap \fS(\upi \times \usi')$ in the analytic topology. Thus the Zariski-closure of this intersection contains all points of $U \cap \ZZ^2$ outside a proper closed subset, and hence must be all of $U$.
  \end{proof}

  Let us write $\cQ(U)$ for the fraction field of $\cO(U)$ (and similarly for $U \times U'$ etc).

  \begin{theorem}
   There exists an isomorphism of $\cQ(U)$-modules
   \[ \ES^2_{\upi}: S^2(\upi) \otimes_{\cO(U)} \cQ(U) \cong \Dcris(\Gr^2 V(\upi)) \otimes_{\cO(U)} \cQ(U), \]
   depending only on $\upi$, characterised uniquely by the following property: for all Hida families $\usi$ as above, and all $P = (r_1, r_2) \in \fS(\upi, \usi)$, the morphism $\ES^2_{\upi}$ is non-singular at $P$ and its fibre at $P$ coincides with the Eichler--Shimura morphism $\ES^2_{\pi_P}$. Moreover, we have the explicit reciprocity law
   \[
    \left\langle
     \cL\left( \mathbf{z}_{m}(\upi \times \usi, \gamma_S)\right),
     \mathrm{ES}^2_{\upi}(\uet) \otimes \mathrm{ES}^0_{\usi}(\underline{\omega})\right\rangle = \cL_{p,\gamma_S}(\upi \times \usi; \uet).
   \]
  \end{theorem}

  \begin{proof}
   We start by choosing a ``random'' isomorphism $\jmath$ between $S^2(\upi)$ and $\Dcris(\Gr^2 V(\upi))$, which is possible since both are free rank 1 $\cO(U)$-modules.

   As in the proof of the preceding lemma, we choose local data $\gamma_S$ such that $\cL_{p, \gamma_S}(\upi, \usi; \uet)$ is not identically zero, and consider the ratio
   \[ \mathsf{R} = \frac{1}{\cL_{p, \gamma_S}(\upi \times \usi, \uet)}  \left\langle\cL\left( \mathbf{z}_{m}(\upi \times \usi, \gamma_S)\right),
        \jmath(\uet) \otimes \mathrm{ES}^0_{\usi}(\underline{\omega})\right\rangle \in \cQ(U \times U').
   \]

   If we now take a $(P, Q)$ that is good geometric, and such that $\cL_{p, \gamma_S}(\upi, \usi; \uet)$ does not vanish at $(P, Q)$, it follows from the \cref{prop:geomreg} that $\mathsf{R}$ is regular at $(P, Q)$ and its value there is equal to the ratio $\jmath_P / \ES^2_{\pi_P}$ (independent of $Q$).

   We claim that $\mathsf{R} \in \cQ(U)$; that is, as a meromorphic function on $U \times U'$, it is independent of the $U'$ variable. To justify this, we argue as in Proposition 17.7.3 of \cite{LZ20}: we consider the meromorphic function $\mathsf{R}(\br_1, \br_2, \bt_2) - \mathsf{R}(\br_1, \br_2, \hat{\bt}_2)$ on $U \times U' \times U'$, where $\hat{\bt}_2$ is the coordinate on a second copy of $U'$. Because of \cref{prop:geomreg}, this function has to vanish at all points $(r_1, r_2, t_2, \hat{t}_2)$ such that $(r_1,r_2, t_2)$ and $(r_1, r_2, \hat{t}_2)$ are both good geometric and neither is in the vanishing locus of $\cL_{p, \gamma_S}(\upi \times \usi, \uet)$; this set is easily seen to be Zariski-dense in $U \times U' \times U'$. The same argument also shows that $\mathsf{R}$ doesn't depend on $\gamma_S$.

   Thus $\mathsf{R}$ is an element of $\cQ(U)^\times$, regular at all points $P \in \fS(\upi, \usi)$ and coinciding at each such point with the ratio $j_P / \ES^2_{\pi_P}$. So if we define $\ES^2_{\upi} = \mathsf{R}^{-1} \jmath$, then $\ES^2_{\upi}$ is regular at all points in $\fS(\upi, \usi)$ and coincides at such points with $\ES^2_{\pi_P}$. By the preceding lemma, this interpolating property uniquely determines $\ES^2_{\upi}$, and is independent of $\usi$; and the reciprocity law holds by construction.
  \end{proof}

  \begin{remark}
   Note that there could, \emph{a priori}, be points where $\ES^2_{\upi}$ is 0 or $\infty$; or where it is a well-defined isomorphism but this isomorphism does not coincide with $\ES_{\pi_P}^2$.
  \end{remark}

 \subsection{Application to the Bloch--Kato conjecture}

  Let us now consider the following situation:
  \begin{itemize}
   \item $\pi$ and $\sigma$ are cohomological cuspidal automorphic representations of $\GSp_4 \times \GL_2$, with $p$-stabilisations which are ordinary and $p$-regular, which are ``deformable'' in the above sense.

   \item If $t_2 = r_1 - r_2 - 1$ (so that $t_1 = -1$), then we suppose that $L(\Pi \times \Sigma, \tfrac{1}{2}) \ne 0$. (In all other cases the non-vanishing of $L(\Pi \times \Sigma, 1+\tfrac{t_1}{2})$ is automatic.)

   \item The Galois representation $V = V_p(\pi)^* \otimes V_p(\sigma)^*(-1-r_1)$ satisfies the ``big image'' conditions of \cite[\S 3.5]{mazurrubin04}.

   \item None of the eight characters appearing as graded pieces of $V$ as a $\Gal(\overline{\QQ}_p / \Qp)$-representation are congruent mod $p$ to the trivial character, or to the $p$-adic cyclotomic character (``$p$-distinction'').
  \end{itemize}

  (Note that the ``big image'' hypothesis can only be satisfied if $\chi_0\chi_2 \ne 1 \bmod p$, but is frequently satisfied when this condition does hold; compare the discussion in \S 11.1 of \cite{KLZ17} in the Rankin--Selberg case.)

  \begin{theorem}
  \label{thm:BKconj}
   In the above setting, we have
   \[ H^1_{\mathrm{f}}(\QQ, V(\pi)^* \otimes V(\sigma)^*(-1-r_1)) = 0,\]
   as predicted by the Bloch--Kato conjecture.
  \end{theorem}

  \begin{proof}
   If $(r_1, r_2)$ is in the set $\fS(\upi, \usi)$ defined above (or more generally in $\fS(\upi, \usi')$ for some possibly different Hida family $\usi'$), then the theorem of the previous section implies that we have an Euler system for $V(\pi)^* \otimes V(\sigma)^*(-1-r_1)$ whose bottom class is non-zero. Hence we may apply the machinery of ``Euler systems with local conditions'' developed in \cite[\S 12]{KLZ17} to deduce the finiteness of the Selmer group.

   The exceptional case which we need to deal with is when the ``family'' Eichler--Shimura isomorphism degenerates at $(r_1, r_2)$. We expect that this never occurs, but we cannot yet rule it out. In this situation, we use a version of the ``leading term argument'' from \cite{LZ20, LZ20-yoshida}). The construction of the $p$-adic $L$-function (and the proof of the reciprocity law) extend immediately to equivariant $p$-adic $L$-functions over $\QQ(\zeta_m)$, for all $m$ coprime to $T$. If the Eichler--Shimura isomorphism degenerates at $(r_1, r_2)$, then not only the class $\mathbf{z}_1(\pi \times \sigma)$, but all the classes $\mathbf{z}_m$, must satisfy the stronger local condition defined by $\cF^{(f)}$; and this forces all the classes to be zero, as in \ and this forces all of the classes to be zero. So we may replace the whole Euler system by its first derivative (in some arbitrarily chosen direction in weight space) and rescale the Eichler--Shimura isomorphism accordingly. Proceeding inductively, we eventually obtain an Euler system with non-trivial bottom class, and the argument proceeds as before.

   (A slight complication here is that in the exceptional case, the Euler system we obtain for $V$ does not necessarily extend to classes over the $p$-cyclotomic tower satisfying the extra-strong local condition $\cF^{(f)}$, since our explicit reciprocity law does not ``see'' the cyclotomic variable. Hence we cannot use the arguments of \cite[\S 12]{KLZ17} to prove the crucial lemma that this local condition is preserved by the passage from Euler to Kolyvagin systems, as these arguments rely on the presence of the $p$-cyclotomic tower. This is the reason for imposing the rather stringent $p$-distinction hypothesis, which allows us to use the alternative, slightly more direct approach given in the appendix of \cite{leiloefflerzerbes14b}, in which the cyclotomic extension is not needed.)
  \end{proof}


\newlength{\bibitemsep}
\setlength{\bibitemsep}{0.75ex plus 0.05ex minus 0.05ex}
\newlength{\bibparskip}
\setlength{\bibparskip}{0pt}
\let\oldthebibliography\thebibliography
\renewcommand\thebibliography[1]{%
 \oldthebibliography{#1}%
 \setlength{\parskip}{\bibparskip}%
 \setlength{\itemsep}{\bibitemsep}%
}

\newcommand{\noopsort}[1]{}

\providecommand{\bysame}{\leavevmode\hbox to3em{\hrulefill}\thinspace}
\providecommand{\MR}[1]{}
\renewcommand{\MR}[1]{%
    MR \href{http://www.ams.org/mathscinet-getitem?mr=#1}{#1}.
}
\providecommand{\href}[2]{#2}
\newcommand{\articlehref}[2]{\href{#1}{#2}}

\end{document}